\documentclass[12pt, a4paper]{amsart}
\usepackage{amssymb}
\usepackage{url}
\usepackage[all]{xy}
\usepackage{enumerate}
\usepackage{fullpage}	
\usepackage{setspace}

\theoremstyle{plain}
\numberwithin{equation}{section}
\newtheorem{thm}{Theorem}[section]
\newtheorem{prop}[thm]{Proposition}
\newtheorem{cor}[thm]{Corollary}
\newtheorem{lem}[thm]{Lemma}

\theoremstyle{definition}
\newtheorem{dfn}[thm]{Definition}
\newtheorem{ex}[thm]{Example}
\newtheorem{rmk}[thm]{Remark}

\def\rank{\mathop{\mathrm{rank}}\nolimits}
\def\dim{\mathop{\mathrm{dim}}\nolimits}
\def\Im{\mathop{\mathrm{Im}}\nolimits}
\def\Ker{\mathop{\mathrm{Ker}}\nolimits}
\def\Cok{\mathop{\mathrm{Cok}}\nolimits}
\def\Hom{\mathop{\mathrm{Hom}}\nolimits}

\def\<{{\langle}}
\def\>{{\rangle}}

\def\+{\mathop{\oplus}\nolimits}

\def\Supp{\mathop{\mathrm{Supp}}\nolimits}

\def\1{\mathop{\mathrm{id}}\nolimits}
\def\Spec{\mathop{\mathrm{Spec}}\nolimits}

\newcommand{\gl}[2]{{\mathsf{gl}\left({#1},  {#2}\right)}}
\newcommand{\HNF}[3]{{
\xymatrix{
0	\ar[r]	&	{#2}_1	\ar[r]\ar[d]	&	{#2}_2	\ar[r]\ar[d]	&	\cdots\ar[r]	&	{#2}_{n-1}	\ar[r]\ar[d]	&	{#2}_n={#1}\ar[d]	\\
			&	{#3}_1\ar@{-->}[ul]	&	{#3}_2\ar@{-->}[ul] &					&	{#3}_{n-1}\ar@{-->}[ul]		&	{#3}_n\ar@{-->}[ul]	
}
}}

\newcommand{\Aut}[1]{{\mathrm{Aut}\,{#1}}}

\newcommand{\Stab}[1]{{\mathrm{Stab}\,{#1}}}

\newcommand{\Dom}[1]{{\mathrm{Dom}({#1}^{-1})}}

\newcommand{\mf}[1]{{\mathfrak{#1}}}
\newcommand{\mi}[1]{{\mathit{#1}}}
\newcommand{\bb}[1]{{\mathbb{#1}}}
\newcommand{\mca}[1]{{\mathcal{#1}}}
\newcommand{\mr}[1]{{\mathrm{#1}}}
\newcommand{\ms}[1]{{\mathsf{#1}}}
\newcommand{\mb}[1]{{\mathbf{#1}}}

\title{Stability conditions and infinitesimal deformation of curves}
\author{Kotaro Kawatani}
\date{\today}

\begin{document}

\begin{abstract}
Let $\mca X$ be an infinitesimal deformation of 
a smooth projective curve $X_0$ over a field. 
We study stability conditions under such deformations 
and show that the derived push-forward functor 
associated with the inclusion $X_0 \to \mca X$ 
induces an isomorphism between 
the space of stability conditions on $\mca X$ and 
that on $X_0$. 
This yields a direct comparison between the deformed 
and undeformed settings. 
As an application, we prove that 
the autoequivalence group $\Aut{\mb D^b(\mca X)}$ 
naturally acts on $\mb D^b(X_0)$, 
providing a link between derived symmetries 
and the deformation structure.
\end{abstract}
\maketitle

\section{Introduction}

Since the 2000s, the study of triangulated categories 
has developed rapidly. 
One of the key invariants in this area is 
the space of stability conditions $\Stab{\mb D}$ 
on a triangulated category $\mb D$ introduced by Bridgeland 
\cite{MR2373143}. 
A notable feature of this invariant is that it carries the structure
of a complex manifold, although it may be empty.

As one might expect, it is very difficult to 
give an explicit description of $\Stab{\mb D}$. 
For instance, when $\mb D = \mb D^b(X)$ is the bounded derived 
category of coherent sheaves on a scheme $X$, 
the space $\Stab{\mb D}$ is well understood only in special cases. 
If $X$ is a smooth projective curve over a field $\mb k$, 
the space $\Stab{\mb D}$ was described by 
Okada \cite{MR2219846} (genus $g=0$), 
Bridgeland \cite{MR2373143} ($g=1$), 
and Macr\'i \cite{MR2335991} ($g>1$). 
The case of 
an irreducible singular curve $X$ of 
arithmetic genus 1 was first studied 
by Burban and Krau\ss\,  \cite{MR2264663}. 
More general cases have been studied in 
Karube \cite{MR4830066} and  Liu \cite{Liu:2025aa}. 
If $X$ is a Noetherian affine scheme of positive dimension, 
then $\Stab{\mb D}$ is empty by \cite{Kawatani_2023}.

We note that, in the existing literature, 
mainly reduced schemes have been considered.
Therefore, in this paper, we turn our attention 
to the non-reduced case, focusing in particular 
on infinitesimal deformations.

Before discussing the main case of interest, 
we first consider a toy model of infinitesimal deformations. 
Namely, let $X$ be the 
affine scheme $\Spec R$ of 
an Artinian local $\mb k$-algebra $(R, \mf m)$ with the residue field $\mb k$. 
We note that $\Spec R$ is an infinitesimal deformation 
of the point $\Spec \mb k$.

\begin{prop}[=Proposition \ref{prop:key-prop}]\label{pre-prop}
Let $(R, \mf m)$ be an Artinian local ring with the residue field 
$\mb k = R/ \mf m$. 
Then the space
$\Stab{\mb D^b(\Spec R)}$ is non-empty and isomorphic to 
$\Stab{\mb D^b(\Spec \mb k)}$. 
\end{prop}

Proposition \ref{pre-prop} suggests 
that the space of stability conditions 
is invariant under infinitesimal deformations of a point.
In view of this, the case of curves is 
a natural next step to consider. 
Motivated by this observation, we state our main result as follows.

\begin{thm}[=Theorems \ref{thm:main1} and \ref{thm:main2}]
	\label{main-thm}
Let $(R, \mf m)$ be an Artinian local $\mb k$-algebra with 
residue field $\mb k$ and 
let $\pi \colon \mca X \to \Spec R$ be an 
infinitesimal deformation of a smooth projective curve $X_{0}$ 
with the closed embedding 
$j \colon X_{0} \hookrightarrow \mca X$. 
Then the derived functor $j_*$ induces an isomorphism 
\[
j_* ^{-1} \colon \Stab{\mb D^b(\mca X)} \cong \Stab{\mb D^b(X_0)}. 
\]
\end{thm}

In particular, the theorem above indicates that 
the existence of stability conditions
depends only on the underlying reduced scheme and is insensitive to
infinitesimal structure.
It is natural to ask whether Theorem \ref{main-thm} 
admits a generalization to singular curves and 
to smooth projective surfaces. 
Our analysis depends on the fact 
that the homological dimension of $\mb D^b(X_0)$ is $1$. 
Therefore, it seems that new ideas are required in 
order to achieve such a generalization.

The proof relies on a detailed analysis 
of the property of objects which could be stable 
for some stability conditions on 
$\mb D^b(\mca X)$. 
The key step is to show 
that stable objects in $\mb D^b (\mca X)$ lie in the image of the 
functor $j_* \colon \mb D^b (X_0) \to \mb D^b(\mca X)$. 
To this end, we focus on a homological property 
of the inclusion $\mr{Coh}(X_0) \subset \mr{Coh}(\mca X)$ 
on the category of coherent sheaves on $X_0$ and on $\mca X$. 
We call such a property ``almost hereditary'' 
(see Definition \ref{dfn:almost-hereditary}). 
Although it is nontrivial that the inclusion 
$\mr{Coh}(X_0) \subset \mr{Coh}(\mca X)$ satisfies the property, 
we show that this indeed holds 
by using deformation theory of locally free sheaves on curves 
(see Proposition \ref{prop:almost-hereditary}). 

Recall that $\Stab{\mb D^b(\bb P^1)}$ 
admits non-geometric stability conditions. 
Because of this phenomenon, the proof of Theorem \ref{main-thm} 
depends on whether the genus satisfies $g=0$ (cf. Theorem \ref{thm:main2}) 
or $g>0$ (cf. Theorem \ref{thm:main1}).

Recall that the autoequivalence group 
$\Aut{\mb D^b(\mca X)}$ acts on $\Stab{\mb D^b(\mca X)}$. 
By Theorem \ref{main-thm}, 
we obtain an induced action of $\Aut{\mb D^b(\mca X)}$ on 
$\Stab{\mb D^b(X_0)}$. 
It is therefore natural to expect 
that the group 
$\Aut{\mb D^b(\mca X)}$ also acts on $\mb D^b(X_0)$. 
In fact we show the following applying a part of Theorem \ref{main-thm}. 

\begin{prop}\label{main-thm2}
We keep the notation of Theorem \ref{main-thm}. 
Then there exists a group homomorphism
\[
\Aut{\mb D^b(\mca X)} \to \Aut{\mb D^b(X_0)}.
\]
\end{prop}

The proof is divided into three cases according to the genus:
$g=0$, $g=1$, and $g>1$. 
In the case $g>0$, we use Theorem \ref{main-thm}.

\begin{rmk}
In the preprint \cite{kawatani2020stability}, the author showed that $\Stab{\mb D^b(\Spec R)}$ is non-empty if and only if $\Spec R$ has dimension zero.
In \cite{Kawatani_2023}, we provided a generalization of the 
“if” direction to the relative setting.
Proposition \ref{pre-prop} was part of 
the main results of \cite{kawatani2020stability}.
Although \cite{Kawatani_2023} was posted on arXiv later, 
it was published earlier.
For this reason, we have substantially revised the preprint 
and present it here as the present paper; 
accordingly, we have decided to 
withdraw \cite{kawatani2020stability} from arXiv.
\end{rmk}

\subsection*{Notations} 
Through the paper, 
we denote by $\mb k$ a fixed field and denote 
by $(R, \mf m)$ an Artinian local $\mb k$-algebra 
with residue field $\mb k$. 

Let $\mb D$ be a triangulated category. 
For objects $E$ and $E$ in $\mb D$, 
the set of morphisms  from $E$ to the shifts 
$F[n]$ of $F$ where $n \in \bb Z$ 
is denoted by $\Hom ^n_{\mb D} (E, F)$ or 
$\Hom_{}^n (E, F)$. 
By convention, $\Hom^0_{\mb D}(E, F)$ is 
nothing but $\Hom_{\mb D}(E,F)$.

If the triangulated category is the 
bounded derived category $\mb D^b(\mca A)$ 
of an abelian category, 
the $i$-th cohomology of an object $E \in \mb D^b(A)$ with respect 
to the standard $t$-structure is denoted by $H^i(E)$.

Suppose that the abelian category $\mca A$ 
is the category of coherent sheaves on a scheme $X$. 
Then we denote by $\mb D^b(X)$ the triangulated $\mb D^b(\mca A)$. 
In particular, if the scheme $X$ is the affine scheme 
$\Spec A$ of a Noetherian ring $A$, 
we simply write $\mb D^b(A)$ instead of $\mb D^b(\Spec A)$.

\subsection*{Acknowledgement}
The author would like to thank Hiroyuki Minamoto and Atsushi Takahashi 
for their helpful advice in carrying out this research, and 
Dogjian Wu for pointing out several typos in an earlier version. 
This work was supported by JSPS KAKENHI Grant Number 21K03212 and 
in part by the Research Institute for Mathematical Sciences (RIMS), 
 Kyoto University.

\section{Almost hereditary subcategory}

\begin{dfn}\label{dfn:almost-hereditary}
Let $\mca A$ be an abelian category. 
A full abelian subcategory $\mca A_0 \subset \mca A$ is 
said to be \textit{almost hereditary} in $\mca A$ if it  satisfies the following 
condition. 
\begin{itemize}
	\item[(AH)] For any objects $E, F \in \mca A_0$, 
	if $\Hom_{\mca A}^0(E, F)=\Hom_{\mca A}^1(E, F)=0$ then 
	the vanishing $\Hom _{\mca A}^p(E, F)=0$ holds for all $p \in \bb Z$. 
\end{itemize}
\end{dfn}

The terminology almost hereditary reflects the idea that, 
although $\mca A$ does not have global dimension $1$, 
it behaves as if it did when we restrict attention 
to the subcategory $\mca A_{0}$.
We introduce two trivial examples. 

\begin{ex}
Suppose that the global dimension of an abelian category 
$\mca A$ is $1$. 
Then $\mca A$ itself is almost hereditary in $\mca A$.

Let $\mca A$ be the category of 
finitely generated  modules
over an Artinian local ring $(R, \mf m)$ and 
let $\mca A_0$ be the full subcategory of 
$R/\mf m$ modules. 
Then $\mca A_0$ is almost hereditary in $\mca A_0$ as follows. 
If objects $E$ and $F$ in $\mca A_0$ satisfy 
the assumption in (AH), 
then, recalling that they are $\mb k$-vector spaces, 
either $E$ or $F$ is zero and hence the conclusion follows. 
\end{ex}

Let $\mca A_0$ be an abelian full subcategory of an abelian category $\mca A$. 
Set a full subcategory $\mb D^b(\mca A)_{\mca A_0}$ 
of $\mb D^b(\mca A)$  by 
\[
\mb D^b (\mca A)_{\mca A_0} := 
\{
E \in \mb D^b (A) \mid H^i(\mca A) \in \mca A_0 (\forall i \in \bb Z)
\}. 
\]

\begin{lem}\label{lem:vanishing-for-HN}
Let $\mca A_0$ be an abelian full subcategory of an 
abelian category $\mca A$. 
Assume that $\mca A_0$ is almost hereditary in $\mca A $. 
Take $E$ and $F$ in $\mb D^b(\mca A)_{\mca A_0}$. 
If  the vanishings 
\[
\Hom_{}^{n}(E, F)=0  \quad (n \leq 0)
\]
hold, 
then  the following holds. 
\begin{enumerate}
	\item $\forall i\in \bb Z, \forall p \in \bb Z$, $\forall n \in \bb Z_{>0}$, 
	$\Hom^p (H^i(E), H^{i-n}(F) )=0$.  
	\item $\forall i \in \bb Z$, $\Hom^0( H^i(E), H^i(F))=0$. 
\end{enumerate}
\end{lem}

\begin{proof}
Recall the spectral sequence 
\begin{equation} \label{spectral1022}
E^{p,q}_2 = 
\bigoplus_{i \in \bb Z} \Hom ^{p}(H^i(E), H^{i+q}(F)) \Rightarrow 
\Hom^{p+q}(E,F)=E^{p+q}. 
\end{equation}
Put 
$m_1= \max \{ i\in \bb Z \mid H^i(E) \neq 0 \}$ and 
$m_2 = \min \{ i \in \bb Z \mid H^j(F) \neq 0 \}$. 
If $m_1 < m_2$,  then nothing to prove. 
Thus it is enough to prove the assertion for $m_1 \geq m_2$ and 
we prove it by the induction for $m_1-m_2$.

Suppose $m_1 = m_2$. 
Then the first assertion is obvious since 
$H^i(E)$ or $H^{i-n}(F)$ is zero. 
Since 
$E_{2}^{p,q}=0$ holds for $q <0$, 
the equality 
$E_{2}^{0,0}=E_{\infty}^{0,0}=E^{0}=\Hom(E, F)=0$ implies 
$E_2^{0,0}=0$.  
Thus we have 
$\Hom^{0}(H^i(E), H^i(F))=0$.

Suppose that the assertions hold for $m_1-m_2-1$. 
By the equlaity 
$ 
\Hom^p(E, F)=\Hom^{p+1}(E, F[-1])
$, 
the pair $(E, F[-1])$ satisfies 
the induction hypothesis. 
Hence, for all $i \in \bb Z $, 
\begin{equation}\label{eq:1022}
\begin{cases}
\forall p \in \bb Z, \forall n \geq 2, 
\Hom^{p}(H^i(E), H^{i-n}(F))=0 \\
\Hom^{0}(H^i(E), H^{i-1}(F))=0. 	
\end{cases}
\end{equation}
Since each cohomology are in $\mca A_0$ 
which is almost hereditary in $\mca A$, 
it is enough to prove 
$\Hom^1(H^i(E), H^{i-1}(F))=0$ and 
$\Hom^0(H^i(E), H^i(F))=0$. 

For the spectral sequence (\ref{spectral1022}), 
the vanishings (\ref{eq:1022}) imply
$E_{\infty}^{1,-1}=E_{2}^{1,-1}$. 
Then the short exact sequence 
\[
\xymatrix{
0 \ar[r] & 
E_{\infty}^{1,-1} \ar[r] & 
E^{0} \ar[r] & 
E_{\infty}^{0,0} \ar[r] & 
0
}
\]
implies 
$E^{1,-1}_2 = 
\Hom^1 ( H^i(E), H^{i-1}(F))=0
$
since $E^0=0$. 
By the assumption for $\mca A_0$, 
we have 
$\Hom ^p (H^i(E), H^{i-1}(F))=0$ for any $p \in \bb Z$. 
Then we see 
\[
E^{0,0}_{\infty} = E^{0,0}_{2} = \Hom (H^i(E), H^i(F))=0. 
\]
This gives the proof. 
\end{proof}

\begin{lem}\label{lem:nilpotent-morphism}
Let $\mca A_0\subset \mca A$ be a full abelian subcategory. 
Suppose that $\mca A_0$ is almost hereditary in $\mca A$.
For a non-zero object $E\in \mb D^b(\mca A)_{\mca A_0}$, put
$m_0 = \max\{ i \in \bb Z \mid H^i(E) \neq 0 \}$. 

If the vanishings 
$
\Hom_{}^{n}(E, E)=0$
hold for $n <0$ and 
$\Hom^1(H^{m_0}(E), H^{m_0-1}(E))$ is non-zero, 
then there is a non-zero morphism 
$\psi  \colon E \to E$ with $\psi^2 =0$. 
\end{lem}

\begin{proof}
By Lemma \ref{lem:vanishing-for-HN}, 
we have 
\begin{equation}\label{eq:0909}
\begin{cases}
\Hom_{}^p (H^i(E), H^{i-q}(E)) =0  & \forall p \in \bb Z , q \geq 2 \\
\Hom_{}(H^i(E), H^{i-1}(E)) =0  & 
\end{cases}
\end{equation}
Without loss of generality, 
we may assume $m_0=0$. 
The cohomological filtration of $E$ with respect to the 
standard $t$-structure is denoted by 
\[
\xymatrix{
\cdots \ar[r]  & E^{-2} \ar[r]\ar[d] & E^{-1} \ar[r]  \ar[d] & E^{0} =E  \ar[d]^{\tau_0}\\
& A^{-2}[2] \ar@{-->}[ul] & A^{-1}[1] \ar@{-->}[ul] & A^0 \ar@{-->}[ul]
}. 
\]
Here $A^i =H^i(E)$.  By the assumption, 
there is a non-zero morphism $\varphi \colon A^0 \to A^{-1}[1]$.

Then the vanishings 
(\ref{eq:0909}) imply 
$\Hom(A^0, E^{-2}[p])=0$ ($\forall p \in \bb Z$). 
Hence we see 
$\Hom(A^0, E^{-1}) \cong \Hom(A^0, A^{-1}[1])$ and 
there exists a lift $\tilde{\varphi}$ to $E^{-1}$ of   
$\varphi \colon A^0 \to A^{-1}[1]$: 
\[
\xymatrix{
&  & A^0 \ar[d]^{\varphi} \ar@{-->}[ld]_{\tilde{\varphi}} & \\
E^{-2} \ar[r]	 & E^{-1} \ar[r]&	 A^{-1}[1]	 \ar[r]	& E^{-2}[1] . 
}
\]
Moreover the morphism $f$ in the diagram below is an isomorphism. 
The morphism $g$ is mono by 
$\Hom(E^{-1}[1], A^{-1}[1])=0$. 
\[
\xymatrix{
&& \Hom^1(E^{-1}[1], A^{-1}) \ar[d] & \\
\Hom(A^0, E^{-2}) \ar[r] & \Hom(A^0, E^{-1}) \ar[r]^{f} \ar[d]^{\tau_0^*} & \Hom^1(A^0, A^{-1})  \ar[r] \ar[d]_{g} & \Hom^1(A^0, E^{-2}) \\
 & \Hom(E^0, E^{-1}) \ar[r] & \Hom^1(E^0, A^{-1}) & \\
}
\]
Since the composite $g \cdot f$ is mono, the morphism 
$\tau_0^*$ is mono and the morphism 
\[
\tau_0^*(\tilde{\varphi})=\tilde{\varphi}\cdot \tau_0
\colon E^0 \to E^{-1}
\]
is non-zero. 

The exact triangle 
\[
\xymatrix{
A^0[-1]		\ar[r]	& E^{-1} \ar[r]^{\rho} & E^0  \ar[r]^{\tau_0} & A^0 , 
}
\]
yields the following exact sequence 
\[
\xymatrix{
\Hom^{-1}(E^0, A^0)  \ar[r] &  \Hom(E^0, E^{-1}) \ar[r]^{\rho_*}& 
\Hom(E^0, E^0). \\
}
\]
Since $\mca A$ is the heart of a $t$-structure, 
$\Hom^{-1}(E^0, A^0)$ is zero 
hence  $\rho_*$ is a mono morphism. 
Thus the morphism 
$
\psi := \rho_*(
\tau_0^* (\tilde{\varphi} )
)
= \rho \cdot \tilde{\varphi}\cdot \tau_0  \colon E^0 \to E^0
$
is non-zero and we easily see that the composite $\psi ^2$ is zero. 
Thus the morphism $\psi$ is desired morphism and this gives 
the proof. 
\end{proof}

\begin{prop}\label{prop:class-of-stable}
Assume that a full subcategory $\mca A_0$ of $\mca A$ 
is almost hereditary and that 
a non-zero object 	
$E \in \mb D^b(\mca A)_{\mca A_0}$
satisfies 
\begin{itemize}
	\item[(a)] $\Hom^{n}(E, E) =0$ for $n <0$ and 
	\item[(b)] any non-zero endomorphism of $E$ is invertible. 
\end{itemize}
Then $E$ is in $\mca A_0$ up to shift. 
\end{prop}

\begin{proof}
If necessary by shift of the complex, 
we may assume that 
the object $E$ satisfies $H^i(E)=0$ for $i>0$ and $H^0(E)\neq 0$.

Let $E^{-1}$ be the co-cone of the truncation morphism $E \to H^0(E)$. 
Then we have the following exact triangle: 
\begin{equation}\label{0912-coffee}
\xymatrix{
E^{-1} \ar[r] & E \ar[r]  & H^0(E) \ar[r] & E^{-1}[1]
}
\end{equation}
It is enough to show that $E^{-1}=0$.

Applying Lemma \ref{lem:vanishing-for-HN} to $F=E[-1]$, 
we have 
\begin{align}
& \Hom ^{0} (H^0(E), H^{-1}(E))=0, \text{ and } \notag  \\
& \Hom^p(H^0(E), H^{-q}(E))=0 \ (\forall p \in \bb Z, \forall q\geq 2). \label{eq:1224-1} 
\end{align}

If $\Hom^1(H^0(E), H^{-1}(E)) \neq 0$, 
then Lemma \ref{lem:nilpotent-morphism} yields a nilpotent 
endomorphism $\psi \colon E \to E$, 
which contradicts the assumption on $E$.
Therefore,
$ \Hom^1(H^0(E), H^{-1}(E)) =0$. 
Since $\mca A_0$ is almost hereditary in $\mca A$, 
it follows that 
\begin{equation} \label{eq:1224-2}
\Hom ^p(H^0(E), H^{-1}(E))=0 \ (\forall p \in \bb Z). 
\end{equation}

Then the vanishings (\ref{eq:1224-1}) and (\ref{eq:1224-2}) imply 
$\Hom^p(H^0(E), E^{-1})=0$ for all $p \in \bb Z$. 
Then $E$ is isomorphic to the direct sum $E^{-1} \+ H^0(A)$. 
By the assumption (b) and \cite[Lemma 3.3]{MR4517994}, 
$E$ is indecomposable.
Hence $E^{-1}=0$.
\end{proof}

\section{Preliminaries on stability conditions}

We first give a quick review of stability conditions on a 
tirangulated category.

\subsection{Review of stability conditions}

Let $\mb D$ be a triangulated category. 
Following the original article \cite{MR2373143}, 
we briefly recall stability conditions on 
a triangulated category $\mb D$. 

\begin{dfn}\label{dfn:stability-condition}
A stability condition on $\mb D$ consists 
of a pair $\sigma=(Z, \mca P)$ 
where $Z$ is a group homomorphism from 
the Grothendieck group of $\mb D$ to $\bb C$ 
and $\mca P=\{ \mca P(\phi) \}_{\phi \in \bb R}$ 
is the collection of full sub-additive 
categories $\mca P(\phi)$ of $\mb D$ 
satisfying the following: 
\begin{enumerate}
	\item[(s1)] any $E \in \mca P(\phi)$ satisfies $Z(E)\in R_{>0} \cdot e^{\pi\sqrt{-1}\phi}$ , 
	\item[(s2)]if $\phi > \psi$ then $\Hom(E,F)=0$ for any 
	$E \in \mca P(\phi)$ and $F \in \mca P(\psi)$, 
	\item[(s3)] $\mca P(\phi +1 ) = \mca P(\phi)[1]$,  and 
	\item[(s4)]any object $E \in \mb D$ has a sequence of exact triangles in $\mb D$
	\begin{equation}
	\label{eq:HN-filtration}	
\xymatrix{
0 \ar[r] & E_1 \ar[d]\ar[r] & E_2 \ar[r] \ar[d] & \cdots \ar[r]  & E_{n-1}\ar[d] \ar[r] & E_n=E \ar[d]  \\ 
		&A_1 \ar@{-->}[ul]&A_2 \ar@{-->}[ul]& & A_{n-1}\ar@{-->}[ul]&A_{n}\ar@{-->}[ul]& 
}	
	\end{equation}
	such that 
	each cone $A_i$ of the morphism $E_{i-1}\to E_i$ is in $\mca P(\phi_i)$ with 
	$\phi_1 > \phi_2 > \cdots > \phi_n$. 
\end{enumerate}
\end{dfn}

\begin{rmk}
\begin{enumerate}
	\item 
An object $A \in \mb D$ is said to be 
\textit{$\sigma$-semistable} if $A$ is in $\mca P(\phi)$ and $A$ is non-zero. 
Moreover the object $A$ is said to be 
\textit{$\sigma$-stable} if $A$ is simple in $\mca P(\phi)$. 
The full subcategory of stable objects in $\mca P(\phi)$ is 
denoted by $\mca P(\phi)^s$.  

\item 
Given an interval $I \subset \bb R$, 
the extension closure 
$\bigcup _{\phi \in I} \mca P(\phi)$ is denoted by 
$\mca P(I)$. 
A stability condition $(Z, \mca P)$ is locally finite if, for any
$\phi \in \mathbb R$, there exists $\delta > 0$ such that
$\mca P((\phi-\delta, \phi+\delta))$ has finite length.

\item The set of locally finite stability conditions on $\mb D$ 
is denoted by $\Stab{\mb D}$. 
The set $\Stab{\mb D}$ has a natural topology and 
a structure of complex manifolds by \cite{MR2373143}.

\item If a stability condition is locally finite, 
any semistable objects is given by finite extension 
of stable objects. 
Thus the slicing $\mca P(\phi)$ is generated by 
$\mca P(\phi)^s$.

\item 
By the definition, 
a stability condition on a triangulated 
category is the pair 
$(Z, \mca P)$ satisfying the above 4 conditions. 

When the pair $(Z, \mca P)$ satisfies (s1), (s2), and (s3), 
if the filtration of an object $E \in \mb D$ exists, 
it is unique up to isomorphisms. 
Hence we refer to the filtration (\ref{eq:HN-filtration}) 
as the Harder-Narasimhan filtration (HN-filtration) of an object $E$. 
\end{enumerate}	
\end{rmk}

\begin{lem}\label{lem:direct-sum}
	Assume that the pair $(Z, \mca P)$ satisfies the conditions (s1), (s2), and (s3), 
	where $Z \colon K_0(\mb D) \to \bb C$ and $\mca P$ is the collection of 
	full sub categories. 
	If $X $ and $Y$ in $\mb D$ have 
	the HN filtrations, then so does $X \+ Y$. 
\end{lem}

\begin{proof}
Let 
$0=X_0 \to X_1 \to \cdots \to X_n= X$ and 
$0=Y_0 \to Y_1 \to \cdots \to Y_m= Y$ 
be the HN filtrations of $X$ and of $Y$ respectively. 
Put by $A_i$ (resp. $B_j$) of the cone 
of $X_{i-1} \to X_i$ (resp. $Y_{j-1} \to Y_j$). 
The phase of $A_i$ (resp. of $B_j$) is 
denoted by $\phi_i$ (resp. $\psi_j$). 
Let $
\{ \theta_1, \theta _2, \cdots, \theta_{\ell} \}= 
{\phi_{i}}_{i=1}^n \cup \{ \psi _j \}_{j=1}^m$ 
with $i<j \Rightarrow \theta _i > \theta _j$.

Then for each $\phi _i$ (resp. $\psi _j$), 
there exists a unique  
$\alpha (i) \in \{ 1, 2, \cdots, \ell \}$
(resp. $\beta (j)\in \{ 1,2, \cdots, \ell \}$)
such that 
$\phi_i = \theta _{\alpha(i)}$
(resp. $\psi_j = \theta _{\beta (j)}$).

Using $\alpha(i)$ and $\beta (j)$, 
we set a filtration $\{ \tilde{X}_{\alpha} \}_{\alpha=1}^{\ell}$ of $X$ by 
\begin{align*}
& \tilde{X} _{\alpha} = X_i \quad (\alpha (i)\leq \alpha < \alpha (i+1))  \\
& \tilde{X} _{\alpha} \to  \tilde {X}_{\alpha +1 } 
=\begin{cases}
	\1 & \alpha (i) \leq \alpha < \alpha (i+1)-1  \\
	X_{i-1} \to X_i & \alpha = \alpha (i+1)-1
\end{cases} 
\end{align*}
Similarly for $Y$, set a filtration 
$\tilde{Y}_{j}$ as follows: 
\begin{align*}
& \tilde{Y} _{\beta} = Y_j \quad (\beta(j)\leq \beta < \beta (j+1))  \\
& \tilde{Y} _{\beta} \to  \tilde {Y}_{\beta +1 } 
=\begin{cases}
	\1 & \beta (i) \leq \beta < \beta (j+1)-1  \\
	Y_{j-1} \to Y_j & \beta = \beta (j+1)-1
\end{cases} 
\end{align*}
Then 
$Z_k = \tilde{X}_k \+ \tilde {Y}_k$ gives 
the HN filtration of $X \+ Y$. 
\end{proof}

Before the introduction of a key lemma, 
we quickly review of a linear structure of 
an abelian category. 
If an abelian category $\mca A$ is 
linear over a commutative ring $K$, 
there exists a multiplication map 
\[
\mu \colon K \to  \mr{End}(E)
\]
for any object $E \in \mca A$. 
Here we say that an object $E \in \mca C$ is 
\textit{annihilated by an ideal $I$} of $K$ if 
$I$ maps to zero via the map $\mu$. 
If any object in a full subcategory $\mca  B$ of $\mca A$ 
is annihilated by an ideal $I$, 
then we call the category $\mca B$ to be annihilated by $I$. 
If the category $\mca B$ is maximal 
among the category annihilated by $I$, 
we call $\mca B$ maximally annihilated by $I$.

\begin{prop}
Let $\mca A$ be an $K$-linear abelian category where 
$K$ is a commutative ring. 
Set a full subcategory $\mca A_I$ by 
\[
\mca A_I := \{ E \in \mca A \mid   E
\text { is annihilated by }I  \}
\]
Then 
$\mca A_I$ is an abelian subcategory and is maximally annihilated. 
\end{prop}

\begin{proof}
The maximality is obvious. 
To be annihilated by $I$ is closed under 
subobjects, quotients, and direct sums. 
Hence $\mca A_I$ is an abelian subcategory.  
\end{proof}

\begin{lem}\label{lem:stable-objects}
Let $\mca A$ be an abelian category 
which is linear over $(R, \mf m)$. 
Assume that
a full abelian subcategory of $\mca A_{\mf m}$ of $\mca A$ 
is almost hereditary.

If an object $E \in \mb D^b(\mca A)$ 
is stable with respect to a stability condition $\sigma \in \Stab{\mb D^b(\mca A)}$, 
then $E$ is indecomposable and in $\mca A_{\mf m}$ up to shifts. 
\end{lem}

\begin{proof}
For any $x \in \mf m$, 
the endomorphism $\mu_x$ of $E$ is nilpotent. 
Since $E$ is stable, any non-zero endomorphism 
is invertible, 
the morphism $\mu_x$ must be zero and each cohomology of $E$ is in 
$\mca A_{\mf m}$.

Thus $E$ is in $\mb D^b(A)_{\mca A_{\mf m}}$. 
By Proposition \ref{prop:class-of-stable}, 
$E$ is in $\mca A_{\mf m}$ up to shifts. 
The indecomposability follows from 
\cite[Lemma 3.3]{MR4517994}. 
\end{proof}

\begin{cor}\label{cor:stable-objects-on-R}
Let $\mb D^b(R)$ be the 
bounded derived category of finitely generated modules over 
an Artinian local ring $(R, \mf m)$. 
For any $\sigma \in \Stab{\mb D^b(R)}$, 
the residue field $\mb k = R/\mb m$ is stable. 
\end{cor}

\begin{proof}
Let $\mca A$ be the abelian category of 
finitely generated modules over $(R, \mf m)$. 
Then the subcategory $\mca A_{\mf m}$ is the 
abelian category 
of finitely generated $\mb k$ modules and 
it is almost hereditary in $\mca A$. 
Thus a stable object must be in $\mca A_{\mf m}$ up to shifts. 
Now any object in $\mca A_{\mf m}$ is isomorphic to 
a direct sum of $\mb k$. 
Hence $\mb k$ must be stable. 
\end{proof}

\subsection{Inducing stability conditions}
An exact functor $F \colon \mb D \to \mb D'$ between triangulated
categories does not, in general, induce a map between the spaces of
stability conditions.
However, a ``good'' functor $F \colon \mb D \to \mb D'$ induces a map,
denoted by $F^{-1}$, from a subset of $\Stab(\mb D')$ to $\Stab(\mb D)$,
as shown by Macr\'i--Mehrotra--Stellari \cite{MR2524593}.
Let us briefly recall the construction of $F^{-1}$.

Let $F \colon \mb D \to \mb D'$ be an exact functor 
between triangulated categories. 
Assume that $F$ satisfies the following condition:
\begin{enumerate}
\item[(Ind)] $\Hom_{\mb D'}(F(a) , F(b))=0$ implies $\Hom_{\mb D}(a,b)=0$ for any $a, b\in \mb D$. 
\end{enumerate}

Let $\sigma' =(Z', \mca P')\in \Stab {\mb D'}$. 
Define $F^{-1}\sigma'$ by the pair $(Z, \mca P)$ where 
\begin{equation}
Z = Z'\circ F, \ \mca P(\phi) = \{ x \in \mb D \mid F(x) \in \mca P'(\phi) \}. 
\end{equation}\label{eq:inducing}
By the definition of $F^{-1}\sigma'$, the pair $F^{-1} \sigma'$ is a stability condition on $\mb D$ if and only if $F^{-1} \sigma' $ has the Harder-Narasimhan property. 

\begin{lem}[{\cite[Lemma 2.9]{MR2524593}}]\label{lm:MMS}
Let $F \colon \mb D \to \mb D'$ be an exact functor 
between triangulated categories satisfying the 
condition (Ind). 
The map $F^{-1} \colon \Dom F \to \Stab {\mb D}$ is continuous. 
\end{lem}

\begin{rmk}
Recall that the universal cover $\widetilde{\mr {GL}}_{2}^{+}(\bb R)$ of 
$\mr{GL}_{2}^{+}(\bb R)$ has the right action to the space of stability conditions. 
The map $F^{-1}$ is $\widetilde{\mr {GL}}_{2}^{+}(\bb R)$-equivariant by the definition of $F^{-1}$. 
\end{rmk}

\begin{prop}\label{prop:injectivity}
Let $\mca A$ be abelian and linear over 
$(R, \mf m)$. 
Assume that 
a full sub-abelian category of $\mca A_0$ of $\mca A$ 
is almost hereditary and maximally annihilated by $\mf m$. 
If the natural functor 
$j_* \colon \mb D^b (\mca A_0) \to \mb D^b (\mca A)$
satisfies the condition (Ind), 
then 
the natural map 
\[
j_*^{-1} \colon \Dom{j_*} \to \Stab{\mb D^b (\mca A_0)}
\]
is injective. 
\end{prop}

\begin{proof}
If $\Dom{j_*}$ is empty, 
then nothing to prove. 
Hence assume $\Dom{j_*} \neq \emptyset$.

Take $\sigma = (Z, \mca P)$ and  $\tau = (W, \mca Q) $ 
in $\Dom{j_*}$ so that 
$j_* ^{-1}\sigma = j_* ^{-1}\tau$. 
By the definition of the map $j_*^{-1}$, 
the slicing $j_*^{-1}\mca P(\phi)$ of $j_*^{-1}\sigma$ is 
just $\mca P(\phi) \cap \Im j_*$ and 
we have 
$j_*^{-1}\mca P(\phi)^s = \mca P(\phi)^s \cap \Im j_*$
On the other hand, 
Lemma \ref{lem:stable-objects} implies 
$\mca P(\phi)^{s} \subset \Im j_*$. 
Hence we see 
\[
j_*^{-1} \mca P(\phi) ^s= \mca P (\phi)^s \cap \Im j_* = \mca P (\phi)^s
\]
Thus $\mca P(\phi)^s=\mca Q(\phi) ^s$ if $j_*^{-1} \sigma = j_*^{-1}\tau$. 
Since $\mca P(\phi)^s$ generates $\mca P(\phi)$, 
we have $\mca P(\phi)=\mca Q(\phi)$. 

Thus an object $E \in \mb D^b(\mca A)$ 
is $\sigma$-stable if any only if it is $\tau$-stable. 
If $j_*^{-1} \sigma = j_*^{-1}\tau$, then 
the value of central charge of any stable object coincides. 
Since the class of stable objects generates 
the group $K_0(\mca D^b(\mca A))$, 
we see that $Z=W$ and $j_*^{-1}$ is injective. 
\end{proof}

\begin{prop}\label{prop:key-prop}
Let $j \colon \Spec \mb k \to \Spec R$ be 
the natural inclusion.  
The bounded derived category of finite $R$-module 
(resp. finite $\mb k$-module)
is denoted by $\mb D^b(R)$ (resp. $\mb k$). 

Then $\Stab{\mb D^b(R)}$ is isomorphic to $\Stab{\mb D^b (\mb k)}$ via 
the natural functor $j_* \colon \mb D^b(\mb k) \to \mb D^b(R)$. 
\end{prop}

\begin{proof}
By Corollary \ref{cor:stable-objects-on-R}, 
we see that $\mb k$ is stable for any $\sigma \in \Stab{\mb D^b(R)}$. 
Hence any finite dimensional vector space is 
semistable for any $\sigma$. 
This implies that $\Dom{j_*} = \Stab{\mb D^b(R)}$.  
Moreover, the action of $\widetilde{\mr{GL}}_2(\bb R)$ is transitive. 
Thus $\Stab{\mb D^b(R)}$ is connected.

By Proposition \ref{prop:injectivity}, 
the natural map 
\[
j_*^{-1} \colon \Stab{\mb D^b(R)} \to \Stab{\mb D^b(\mb k)}
\]
is injective.

It is immediate the map is surjective. 
In fact, for any $\tau= (W, \mca Q) \in \Stab{\mb D^b(\mb k)}$, 
$\mb k$ is stable with phase $\psi$. 
Since the Grothendieck group of the abelian category
$\ms{mod}\, R$ of finitely generated $R$-modules is isomorphic 
to the Grothendieck group of $\ms{mod}\, \mb k$, 
the central charge $W$
induces a central charge on $\mb D^b(R)$.
Moreover, since 
any finite $R$-module is semistable, 
the pair $\sigma= (W, \ms{mod}\, R )$ gives 
a stability condition on $\mb D^b(R)$ and 
it clearly satisfies $j_*^{-1} \sigma = \tau$.

Since both spaces $\Stab{\mb D^b(R)}$ and 
$\Stab{\mb D^b}(\mb k)$ are complex manifolds, 
they are isomorphic to each other. 
\end{proof}

\begin{thm}\label{thm:affine-case}
Let $A$ be a Noetherian ring. 
The space $\Stab {\mb D^b (A)}$ is non-empty 
if and only if the Krull dimension $\dim A$ of $A$ is zero. 
Moreover, if $\dim A=0$, then 
the space $\Stab{\mb D^b(A)}$ is isomorphic to 
$\bb C^{n}$ where $n$ is the number of points of $\Spec A$. 
\end{thm} 

\begin{proof}
Suppose that $\dim A >0$. 
Then the derived category $\mb D^b(A)$ is finite over $A$. 
Hence the result \cite{Kawatani_2023} implies 
$\Stab{\mb D^b (A)}$ is empty.

Suppose that $\dim A= 0$. 
Then $A$ is a finite product of local Artinian ring $\{ R_i \}_{i=1}^n$. 
Hence the derived category $\mb D^b(A)$ 
orthogonally decomposes into the derived category 
of Artinian local rings: 
\[
\mb D^b(A) \cong \bigoplus _{i=1}^n \mb D^b (R_i)
\]
Then the result due to \cite{MR3984103} implies that 
$\Stab{\mb D^b(A)} $ is a product of 
$\Stab{\mb D^b (R_i)}$: 
\[
\Stab{\mb D^b(A)} \cong 
\prod_{i=1}^n \Stab{\mb D^b (R_i)}
\]

Recall that $\Stab{\mb D^b (\mb k) }$ is isomorphic to 
$\bb C$ by \cite[Corollary 3.8]{MR4517994}. 
Hence $\Stab{\mb D^b(R)}$ is also isomorphic to $\bb C$. 
This gives the proof of the desired assertion. 
\end{proof}

Combining the results from \cite{Kawatani_2023}, 
we obtain the following;

\begin{cor}
Let $A$ be a Noetherian ring. 
Then the space $\Stab{\mb D^b(A)}$ is non-empty 
if and only if $\dim A=0$. 
\end{cor}

\begin{proof}
The \textit{if} part follows from the above theorem. 
The \textit{only if}  part follows from \cite{Kawatani_2023}. 
\end{proof}

\subsection{Semiorthogonal decompositions and stability conditions}
Collins and Polishchuck \cite{MR2721656} proposed a 
construction of stability conditions on a 
triangulated category $\mb D$ 
from a semiorthogonal decomposition. 
A key ingredient of the construction is 
a \textit{reasonable} stability condition on a triangulated category.   

\begin{dfn}[{\cite[pp. 568]{MR2721656}}]\label{dfn-reasonable}
A stability condition $\sigma =(\mca A, Z)$ on 
a triangulated category $\mb D$ is 
\textit{reasonable} if $\sigma$ satisfies
\[
0 < \inf	\{	|Z(E)| \in \bb R  \mid 	E\mbox{ is semistable in }\sigma  \}. 
\]
\end{dfn}

\begin{rmk}\label{rmk:reasonable}
A reasonable stability condition is locally finite 
by \cite[Lemma 1.1]{MR2721656}. 
\end{rmk}

Let $\mb D$ be a triangulated category. 
Recall that a pair $(\mb D_{1}, \mb D_{2})$ of 
full triangulated subcategories of $\mb D$ 
is said to be a \textit{semiorthogonal decomposition} of $\mb D$ 
if the pair satisfies 
\begin{itemize}
\item[(1)] $\Hom_{\mb D}(E_{2}, E_{1})=0$ for any $E_{i } \in \mb D_{i}$ ($i=1,2$), and 
\item[(2)] any object $E \in \mb D$ is decomposed into a pair of objects $E_{i}\in \mb D_{i}$ $(i=1,2)$ by the following distinguished triangle in $\mb D$:
\[
\xymatrix{
E_{2}	\ar[r]	&	E	\ar[r]	&	E_{1}	\ar[r]	&	E_{2}[1]. 
}
\]
\end{itemize}
The situation will be denoted by the symbol $\mb D=\< \mb D_{1}, \mb D_{2}\>$ or simply $\< \mb D_{1}, \mb D_{2}\>$. 
In addition to the first condition (1) above, if $\Hom_{\mb D}(E_{1}, E_{2})=0$ holds, the semiorthogonal decomposition is said to be \textit{orthogonal}. 
Note the assignment of $E_i$ gives a functor $\mb D \to \mb D_i$ ($i=1,2$). 
The functor $\tau_1$ is the left adjoint of the inclusion $\mb D_1 \hookrightarrow \mb D$, 
and $\tau_2$ is the right adjoint of $\mb D_2 \hookrightarrow \mb D$.

\begin{dfn}
Let $\mb D=\< \mb D_1, \mb D_2 \>$ be a semiorthogonal decomposition of $\mb D$ and 
let $\sigma _i =(Z_i, \mca A_i) $ be a locally finite stability condition on 
the subcategory $\mb D_i$ ($i=1,2$). 
If a stability condition $\sigma =(Z, \mca A)$ on $\mb D$ satisfies 
\begin{itemize}
	\item $Z(E) = Z_1(\tau_1 E) + Z_2 (\tau_2 E)$, and 
	\item $\mca A = \{   E \in \mb D \mid \tau_i(E) \in \mca A_i\,  (i=1,2) \}$, 
\end{itemize}
then $\sigma$ is said to be \textit{glued from $\sigma_1$ and $\sigma_2$} and 
denoted by $\gl{\sigma_1}{\sigma_2}$. 
\end{dfn}

\begin{prop}[{\cite{MR2721656}}]\label{CP2.2}
Let $\<\mb D_1, \mb D_2  \>$ be a semiorthogonal 
decomposition of a triangulated category $\mb D$ and 
let $\sigma_i = (Z_i, \mca P_i)$ be a reasonable 
stability condition on $\mb D_i$ for $i \in \{1,2\}$.  
Suppose that $\sigma_1$ and $\sigma_2$ satisfy the following conditions
 \begin{enumerate}
\item  $\Hom_{\mb D}^{\leq 0}\left( \mca P_1(0,1], \mca P _2(0,1]\right) =0$ \label{condition-a} and 
\item There is a real number $a\in (0,1)$ such that $\Hom_{\mb D}^{\leq 0} \left(\mca P_1(a,a+1], \mca P _2(a,a+1]\right) =0$. \label{condition-b}
\end{enumerate}

Then there exists a unique reasonable 
stability condition $\gl{\sigma_1} {\sigma_2}$ on $\mb D$ 
glued from $\sigma _1$ and $\sigma _2$. 
\end{prop}

\section{Infinitesimal deformation of curves}\label{sec:4}

We first give an example of almost hereditary subcategory 
of an abelian category. 
Such an example naturally occurs in an infinitesimal deformation 
of curves. 
In this section, 
a smooth projective curve over the field $\mb k$ is denoted by 
$X_0$ and 
an infinitesimal 
deformation of $X_0$ 
is denoted by 
$\pi \colon \mca X \to \Spec R$. 
That is 
$\pi \colon \mca X \to \Spec R$ is flat 
together with 
the isomorphism 
\[
\mca X \times _{\Spec R} \Spec R/\mf m  \cong X_0. 
\]
We first show that 
the abelian category $\mr{Coh}(X_0)$ is almost hereditary in $\mr{Coh}(\mca X)$. 
After that 
we study the space of stability conditions on 
$\mb D^b(\mca X)$. 

\subsection{A review of deformation theory of sheaves}

\begin{lem}\label{lem:deformation-of-vect}
For any locally free sheaf $E_0$ on $X_0$, 
there exists a locally free sheaf $\mca X$ such that 
the restriction to $X_0$ is isomorphic to $E_0$. 
\end{lem}

\begin{proof}
Since
$X_0$ is a smooth projective curve, 
there exists an infinitesimal deformation 
of $X_0$ on an arbitrary Artinian local ring.

	Let $A'$ be a small extension of an 
Artinian local ring $A$, 
that is, suppose that there is a short exact sequence 
\[
\xymatrix{
0 \ar[r] & (t) \ar[r] & A' \ar[r]  & A \ar[r] & 0
} 
\]
with $m_{A'} \cdot (t)=0$ and $A'/\mf m \cong A/ \mf m=\mb k$. 
Let $\mca X_{A'}$ (resp. $\mca X_{A}$) be 
an infinitesimal deformation on $A'$ (resp. $A$). 
Assume that there exists a locally free sheaf $E$ on $\mca X_{A}$
such that $E|_{X_0} \cong E_0$. 
Then the obstruction to lifting of $E$ to $\mca X_{A'}$
is in $H^2(X_0, \mca E\mi{nd}(E_0))$ (for instance see \cite[Theorem 7.1]{MR2583634})
Since $X_0$ is a curve, 
the obstruction vanishes. 
Thus there is a lift of $E$ to $\mca X_{A'}$
\end{proof}

\begin{lem}\label{lem:0220}
Let $j \colon X_0 \to \mca X$ be the closed embedding. 
Put $d_n = \dim _{\mb k}\mr{Tor}_n^R(\mb k,\mb k)$. 
Then the following holds. 
\begin{enumerate}
	\item Let $E_0 $ be a locally free sheaf on $X_0$. 
	For any coherent sheaf $F_0$ on $X_0$, we have 
	\[
	\mca Ext^p _{\mca X}(j_* E_0, j_* F_0) 
		\cong j_* \mca Hom_{X_0}(E_0, F_0)^{\+ d_p}
	\]
	\item 
	Let $T_x$	be an indecomposable torsion sheaf supported 
	on a closed point $x \in X_0$. 
	For any coherent sheaf $F_0$ on $X_0$, we have 
\[
	\mca Ext^p_{\mca X}	(j_* T_x, j_* F_0) \cong 
	j_* \mca Ext^1_{X_0}(T_x, F_0)^{\+ d_{p-1}} \+ 
	j_* \mca Hom_{X_0}(T_x, F_0)^{\+ d_p}
\]	
\end{enumerate}
\end{lem}

\begin{proof}
By Lemma \ref{lem:deformation-of-vect}, 
there exists a locally free sheaf $E$ on $\mca X$ such that 
the restriction to $X_0$ is isomorphic to $E_0$. 
Using $E$, we can take a locally free resolution of $j_* E_0$. 
In fact, take a minimal free resolution of $\mb k$ by $R$: 
\[
\cdots \to R^{\+ d_2} \to R^{\+ d_1} \to R \to \mb k \to 0. 
\]
By the functor $\otimes _{R} E$, 
we obtain a locally free resolution $j_* E_0$ as follows 
\[
\cdots \to E^{\+ d_2} \to E^{\+ d_1} \to E \to \mb k \otimes _{R} E \cong j_*E_0 \to 0. 
\] 

Taking $\mca Hom_{\mca O_{\mca X}}(-,j_* F_0)$, 
we see that
\begin{equation*}
\mca Ext^{p}(j_* E_0, j_* F_0) \cong \mca Hom(j_*E_0, j_* F_0)^{\+ d_p} 
 \cong j_* \mca Hom_{X_0}(E_0, F_0)^{\+ d_p}
 \text{ (for $p \in \bb Z_{\geq 0}$)}  
\end{equation*}

To prove the assertion (2), 
we need a locally free resolution of $j_* T_x$. 
Recall that an affine smooth scheme is rigid. 
Choose an affine open neighborhood $U = \Spec A$ 
of the point $x \in X_0$ so that 
$\mca X|_{U} \cong U \times _{\mb k} \Spec R$. 
If $U$ is sufficiently small, 
then there exists $f \in A$ such that 
$T_x \cong A/(f)$. 
A locally free resolution of $j_* T_x$ 
is given by the following complex: 
\begin{align*}
j_* T_x &\cong T_x \otimes _k R / \mf m \\  
 & \cong 
(A \overset{f}{\to} A ) \otimes _k ( \cdots \to R^{d_2} \to R^{d_1} \overset{\pi}{\to} R) 
\end{align*}
Taking $\mca Hom_{\mca X}(-, j_* F_0)$, 
we obtain the complex 
$\{ \mca Hom_{\mca X}(j_* T_x, j_* F)^{p}, \partial^p \}_{p\in \bb Z}$. 
The complex is explicitly given by 
\begin{align*}
&\mca Hom_{\mca X}(j_* T_x, j_* F_0)^{p} \cong 
\mca Hom_{\mca X}(A \otimes _{\mb k} R, j_* F_0)^{\+ d_{p-1}} 
\+ 
\mca Hom_{\mca X}(A \otimes _{\mb k} R, j_* F_0)^{\+ d_{p}}  \\
& \partial^p \colon  
\mca Hom_{\mca X}(j_* T_x, j_* F_0)^{p} 
\to \mca Hom_{\mca X}(j_* T_x, j_* F_0)^{p+1} , 
\begin{bmatrix}
	x \\ y 
\end{bmatrix} \mapsto 
\begin{bmatrix}
0 &  \mu_{f\otimes 1}^{\+d_p}  \\ 
0 & 0 
\end{bmatrix}
\begin{bmatrix}
x \\ y
\end{bmatrix}. 
\end{align*}
Here we extend the definition of 
$d_p$ to all integers $p$ 
by setting $d_p=0$ for $p<0$, 
and the morphism $\mu_{f\otimes 1}$ is the multiplication of $f\otimes 1$: 
\[
\mu_{f \otimes 1} \colon \mca Hom_{\mca X}( A\otimes _{\mb k} R , j_* F_0) 
\to 
\mca Hom_{\mca X}( A\otimes _{\mb k} R , j_* F_0). 
\]
Thus we obtain 
\begin{align*}
	\mca Ext_{\mca X}^p(j_* \mca O_x, j_* F_0) \cong 
	\begin{cases}
	\Ker (\mu_{f \otimes 1}) & (p=0 ) 	 \\
	\Cok(\mu_{f \otimes 1})^{\+ d_{p-1}} \+ \Ker (\mu_f)^{\+ d_{p}} & (p>0). 
	\end{cases}
\end{align*}

To complete the proof, 
we wish to calculate $\Ker (\mu _{f \otimes 1})$ and 
$\Cok (\mu_{f \otimes 1})$. 
Here we note that 
$\mca Hom_{\mca X}(A\otimes _{\mb k} R, j_* F) \cong j_* \mca Hom_{X_0}(A, F)$. 
Hence the morphism $\mu_{f \otimes 1}$ is nothing but 
the push-forward by the closed embedding $j$ of the multiplication 
of $f$: 
\[
\mu_{f \otimes 1} = j_*(\mu_f \colon 
\mca Hom_{X_0}(A, F) \to \mca Hom_{X_0}(A, F)  )
\]
Since $f$ is non-zero divisor, 
we have 
\begin{align*}
	\Ker (\mu _{f \otimes 1}) & \cong j_* \mca Hom_{X_0}(A/(f) , F_0 ) 
	\text{, and} \\  
	\Cok(\mu_{f \otimes 1}) & \cong j_* \mca Ext_{X_0}^1(A/(f), F_0 ). 
\end{align*}
and complete the proof. 
\end{proof}

\begin{lem}\label{lem:260308}
Put $d_n = \dim _{\mb k}\mr{Tor}_n^R(\mb k,\mb k)$. 
Assume that $E_0$ is an indecomposable sheaf on $X_0$. 
For any coherent sheaf $F_0$ on $X_0$, 
we have 
\[
\Hom_{\mca X}^p(j_* E_0, j_* F_0)\cong 
\Hom_{X_0}^0(E_0, F_0)^{\+ d_p} \+ 
\Hom_{X_0}^1(E_0, F_0)^{\+ d_{p-1}}
\]
\end{lem}

\begin{proof}
Since the sheaves $E_0$ and $F_0$ are defined over 
the field $\mb k$, 
the set 
$\Hom_{\mca X}^p(j_* E_0, j_* F_0)$ is a $\mb k$-module. 
The local-to-global spectral yields the following decomposition: 
\begin{equation}
\label{eq:260308}	
\Hom_{\mca X}^p(j_* E_0, j_* F_0)\cong 
H^0(\mca X, 
\mca Ext_{\mca X}^p(j_* E_0, j_* F_0) 
) 
\+ 
H^1(\mca X, 
\mca Ext_{\mca X}^{p-1}(j_* E_0, j_* F_0))
\end{equation}

By the assumption for $X_0$, 
$E_0$ must be a locally free sheaf on $X_0$ or 
a torsion sheaf $T_x$ supported on a point $x \in X_0$. 
If $E_0$ is locally free, by Lemma \ref{lem:0220}, we have 
\begin{equation}
\Hom_{\mca X}^p(j_* E_0, j_* F_0)\cong 
H^0(X_0, \mca Hom_{X_0}(E_0, F_0))^{\+ d_p} 
\+ 
H^1(X_0, \mca Hom_{X_0}(E_0, F_0))^{\+ d_{p-1}}. 
\end{equation}
By the assumption for $E_0$, 
the second term in the right hand side is isomorphic to 
$\Hom_{X_0}^1(E_0, F_0)^{\+ d_{p-1}}$. 

To complete the proof, 
assume $E_0=T_x$. 
Then Lemma \ref{lem:0220} yields 
\begin{align*}
\Hom_{\mca X}^p(j_* T_x, j_* F_0) &\cong 
H^0(\mca X, \mca Ext_{\mca X}^p(j_* E_0, j_* F_0))  \\
		& \cong H^0(\mca X,  \mca Hom_{X_0}(T_x, F_0))^{\+ d_p} \+ 
		H^0(X_0, \mca Ext^1_{X_0}(T_x, F_0))^{\+ d_{p-1}} \\
		&\cong 
		\Hom_{X_0}(T_x, F_0)^{\+ d_p} \+ \Hom_{X_0}^1(T_x, F_0)^{\+ d_{p-1}}. 
\end{align*}
\end{proof}

\begin{prop}\label{prop:almost-hereditary}
The abelian category $\mr{Coh}(X_0)$ is 
almost hereditary in $\mr {Coh}(\mca X)$. 
\end{prop}

\begin{proof}
The closed immersion $X_0 \to \mca X$ is denoted by $j \colon X_0 \to \mca X$. 
Suppose that $E_0, F_0 \in \mr {Coh}(X_0)$ satisfies 
\begin{equation}\label{eq:0220ass}
\Hom^0_{\mca X} (j_*  E_0, j_*  F_0) = 
\Hom^1_{\mca X} (j_*  E_0, j_*  F_0) =  0. 
\end{equation} 
Without loss of generality, we may assume that 
$E_0$ is indecomposable. 
Then Lemma \ref{lem:260308} yields the desired assertion. 
\end{proof}

\subsection{Stability conditions on $\mca X$}

\begin{prop}
Let $\pi \colon \mca X \to \Spec \, R$ be an 
infinitesimal deformation of $X_0$ 
over $(R, \mf m)$. 
Then there is a locally finite stability condition 
on $\mb D^b (\mca X)$. 
\end{prop}

\begin{proof}
By a theorem due to Quillen \cite{MR338129}, 
the $K$-theory of $\mr{Coh}(\mca X)$ 
is homotopy equivalent to that of $\mr{Coh}(X_0)$ via 
the canonical inclusion $X_0 \to \mca X$. 
In particular, 
the $K_0$ group of $\mr{Coh}(\mca X)$ and of 
$\mr{Coh}(X_0)$ are isomorphic. 

Thus we can take the same central charge of 
the standard stability condition on $X_0$. 
Namely set 
a group homomorphism  
\[
Z_{\mr{st}} \colon 
K_0 (\mr{Coh}(\mca X)) \to \bb C
\]
by 
$Z_{\mr{st}} ([E]) : =  -\deg E|_{X_0} + \sqrt{-1} \rank E|_{X_0}$.

Then the image of $Z_{\mr{st}}$ is clearly discrete. 
Since the category $\mr{Coh}(\mca X)$ is Noetherian, 
the argument in \cite{MR2376815} implies that 
the pair $(Z_{\mr{st}}, \mr{Coh}(\mca X))$ is a 
locally finite stability condition. 
\end{proof}

\begin{dfn}\label{dfn:standard-stability}
We denote by $\sigma_{\mr{st}}$ 
the stability condition constructed in the proof of the above 
proposition: 	
\[
\sigma_{\mr{st}}  = (Z_{\mr {st}}, \mr{Coh}(\mca X)). 
\]
We call $\sigma_{\mr{st}}$ a standard stability condition on $\mca X$. 
\end{dfn}

\begin{rmk}
Since $X_0$ is a smooth curve, 
any object in $\mb D^b(X_0)$ is a direct sum of 
shifts of sheaves on $X_0$. 
Hence the natural functor 
\[
j_* \colon \mb D^b (X_0) \to \mb D^b(\mca X)
\]
is faithful, but not full. 
Thus we obtain the map 
\[
j_* ^{-1} \colon \Dom{ j_*} \to \Stab{\mb D^b(X_0)}. 
\]
The following lemma claims the non-emptiness of $\Dom{j_* }$. 
\end{rmk}

\begin{lem}\label{lem:non-emp-of-Dom}
The standard stability condition $\sigma_{\mr{st}}$
on $\mca X$ belongs 
to 
$\Dom{j_*}$ where $j_* \colon \mb D^b(X_0) \to \mb D^b(\mca X)$. 
\end{lem}

\begin{proof}
Let $\mca P=\{ \mca P(\phi)\}$ be the slicing 
of $\sigma_{\mr{st}}$. 
It is enough to show that 
$\Im j_* \cap \mca P$ satisfies the condition (s4).

Since indecomposable objects in $\Im j_*$
are isomorphic to sheaves on $X_0$, 
it is enough to show the assertion for 
sheaves on $X_0$ by Lemma \ref{lem:direct-sum}. 

The subcategory $\mr{Coh}(X_0)$ is closed under 
subobjects and quotients. 
Hence each semistable factor is in $\mr{Coh}(X_0)$ and 
$j_*^{-1} \sigma$ satisfies the condition (s4). 
\end{proof}

\begin{prop}\label{prop:key-1}
Let $j \colon X_0 \to \mca X$ be the closed embedding. 	
Assume that a stability condition $\sigma$ on $\mb D^b(\mca X)$ satisfies 
\begin{enumerate}
	\item for any point $x \in X_0$, $j_* \mca O_x$ is stable, and 
	\item for any line bundle $L$ on $X_0$, $j_* L$ is stable. 
\end{enumerate}
Then $\sigma $ is in the orbit $ \widetilde{\mr{GL}}_2^+ (\bb R)  \cdot \sigma_{\mr{st}}$. 
\end{prop}

\begin{proof}
Let $\phi_x$ (resp. $\phi_L$) be the phase of $j_* \mca O_x$ (resp. $j_* L $). 
Then the non-vanishings
$\Hom_{\mca X}(j_*L, \mca O_x)\neq 0$ and 
$\Hom_{\mca X}(\mca O_x, j_* L[1])\neq 0$ 
imply 
$\phi_L \leq \phi_x$ and $\phi_x \leq \phi _L +1$. 
Since both of them are stable, 
the equality in the above inequality cannot occur. 
Thus we obtain 
\[
\phi_x -1 <\phi_L < \phi_x < \phi _L +1. 
\]
When we take $L= \mca O_{X_0}$, 
the phase $\phi$ is uniquely determined, 
independent of the choice of a closed point. 
Hence, by the action of 
$\widetilde{\mr{GL}}_2(\bb R)$,
we may assume the following: 
\begin{itemize}
	\item For $\sigma  = (Z, \mca P) \in \Stab{\mb D^b(\mca X)}$, 
	the structure sheaf $\mca O_x$ of any closed point $x \in X_0$ 
	is stable with phase $1$ and $Z(\mca O_x)= -1$.  
	Any invertible sheaf $L$ on $X_0$ is $\sigma$-stable with phase 
	$0 < \psi <1$ and $Z(j_* \mca O_{X_0})=\sqrt{-1}$. 
\end{itemize}

Under the above assumption, 
the central charge of $\sigma$ coincides with that of 
$\sigma_{\mr{st}}$ in Definition \ref{dfn:standard-stability}. 
Hence it is enough to 
show that the $t$-structure $\mca P((0,1])$ of $\sigma $
is contained in $\mr {Coh}(\mca X)$. 
Let $E$ be a $\sigma$-semistable object 
with phase $\phi$ and $0< \phi \leq 1$. 
We show that $E \in \mr{Coh}(\mca X)$.

Let $H^i(E)$ be the $i$-th cohomology with 
respect to the standard $t$-structure of $\mb D^b(\mca X)$. 
Set $m_+ := \max \{ i \mid H^i(E)\neq 0 \}$.  
Then, truncating $E$ with respect to the $t$-structure, 
we obtain the following exact triangle 
\[
\xymatrix{
\tau_{\leq -1} (E[m_+]) \ar[r] & E[m_+] \ar[r] & \tau_{\geq 0}(E[m_+]) \ar[r] 
&\tau_{\leq -1} (E[m_+]) [1]. 
}
\]
Note that $\tau_{\geq 0}(E[m_+]) \cong H^{m_+}(E)\neq 0$. 
Hence we see that there exists a closed point $y$ in $X_0$ such that 
\[
\Hom_{\mca X} (  \tau_{\geq 0}(E[m_+]), \mca O_y) \neq 0. 
\]

If $m_+ >0$, then 
we have 
\[
\Hom_{\mca X} (  \tau_{\geq -1 }(E[m_+]), \mca O_x) \cong 
\Hom_{\mca X} (  \tau_{\geq -1 }(E[m_+])[1], \mca O_x) \cong 
 0. 
\]
Therefore we have 
\begin{equation}\label{eq:1029-16}
\Hom_{\mca X} (E[m_+], \mca O_x) \cong  
\Hom_{\mca X}(\tau _{\geq 0}(E[m_+]), \mca O_x) \neq 0. 
\end{equation}
By the assumptions for $E$ and $m_+$, 
the phase of $E[m_+]$ satisfies 
\begin{equation*}
1 \leq m_+ < \phi + m_+. 
\end{equation*}
Hence 
we have 
$\Hom_{\mca X}(E[m_+], \mca O_y) =0$ which contradicts with 
(\ref{eq:1029-16}). 
Thus we see $m_+ \leq 0$.

Set 
$m_- := \min\{ i \in \bb Z  \mid H^i(E) \neq 0 \}$. 
We finally show that $m_- \geq 0$.   
Let $L_0$ be an ample line bundle on $X_0$. 
We see that there exists an invertible sheaf $\mca L$ 
on $\mca X$ such that 
the restriction to $X_0$ is isomorphic to $L_0$. 

Since $X_0$ is the reduced part of the scheme $\mca X$, 
$\mca L$ is ample on $\mca X$. 
Moreover $\mca L$ is given by the successive extension 
(as $\mca O_{\mca X}$-module) of $L_0$. 
Hence $\mca L$ is $\sigma$-semistable with the 
same phase of $j_* L_0$.

Similarly to the case of $m_+$, 
truncating $E$, we obtain the following exact 
triangle in $\mb D^b(\mca X)$
\[
\xymatrix{
\tau _{\geq 1} (E[m_-])[-1] \ar[r] & 
\tau _{\leq 0}(E[m_-]) \ar[r] & E[m_-] \ar[r] & \tau _{\geq 1} (E[m_-]) . 
}
\]
Note that $\tau _{\leq 0}(E[m_-])\cong H^{m_-}(E) $. 

Set $F= H^{m_-}(E)$. 
Since $\mca L$ is ample on $\mca X$, 
for sufficiently large $n >0$, 
the canonical morphism 
\[
H^0(\mca X, F \otimes \mca L ^{n}) 
\otimes _{R} \mca L^{-n} \to F
\]
is surjective. 
Put $\mca M = H^0(\mca X, F \otimes \mca L ^{n})  \otimes_{R} \mca L^{-n}$, 
and we have 
\begin{equation}\label{eq:1029-nv}
\Hom_{\mca X}(\mca M, F) \neq 0. 
\end{equation}
Since the isomorphisms 
\[
\Hom_{\mca X}(\mca M , \tau _{\geq 1} (E[m_-])[-1])
\cong 
\Hom_{\mca X}(\mca M, \tau _{\geq 1} (E[m_-])) =0
\]
holds, the non-vanishing (\ref{eq:1029-nv}) implies 
\begin{equation}\label{eq:1029-nv2}
0 \neq 
\Hom_{\mca X}( \mca M, E[m_-]) . 
\end{equation}

Note that the sheaf $\mca M$
is given by 
a successive extension of $\mca L^{-n} \otimes _{R} \mb k \cong L_0^{-n}$, 
and it is semistable with the same phase as that of $L^{-n}_0$, 
since the $R$-module $H^0(\mca X, F \otimes \mca L^n)$ is finite and 
$\mca L$ is flat over $R$. 
By the assumption for $E$, 
the phase $\phi + m_-$ of $E[m_-]$ satisfies 
\[
m_- < \phi + m_- \leq 1 + m_-
\]
If $m_- <0$, then 
the inequality $\phi + m_- \leq 0$ holds. 
Since the phase of $\mca M$ is in the interval $(0,1)$, 
we must have 
\[
\Hom_{\mca X}( \mca M, E[m_-])=0
\]
which contradicts with (\ref{eq:1029-nv2}). 
Thus we have $m_- \geq 0$ and this 
gives the proof. 
\end{proof}

\subsection{The case $g(X_0) >0$}

\begin{lem}\label{lem:GKR}
Assume $g(X_0) >0$. 
Let 
	\[
\xymatrix{
	A \ar[r] & E \ar[r] & B \ar[r] & A [1]
}
\]
be an exact triangle in $\mb D^b(\mca X)$ 
with $\Hom_{\mca X}^{\leq 0}(A,B)=0$.

If the complex $E$ concentrates in degree $0$ and 
is an $\mca  O_{X_0}$-module, then 
the same holds for $A$ and $B$. 
\end{lem}

\begin{proof}
By the assumption, 
the multiplication morphism $\mu_{\epsilon}$, 
where $\epsilon \in \mf m$,  of $E$ is zero. 
Then, by \cite[Lemma 2.2]{Kawatani_2023}, 
the multiplications of $A$ and of $B$ are zero. 
Thus 
each cohomology of $A$ and $B$ are $\mca O_{X_0}$ modules. 

Since $\mr{Coh}(X_0)$ is almost hereditary in $\mr{Coh}(\mca X)$ 
by Proposition \ref{prop:almost-hereditary}, 
Lemma \ref{lem:vanishing-for-HN} 
implies the following: 
\begin{equation}
\label{eq:1029-13}	
\Hom_{ \mca X}^p (H^i(A), H^{i-1}(B)) = 0 (\forall p \in \bb Z, i \in \bb Z)
\end{equation}
If $i \in \bb Z -\{ 0,1 \}$, 
then 
$H^i(A) \cong H^{i-1}(B)$. 
Hence we have
\[
\begin{cases}
	H^i(A) = 0 & i \neq 0, 1 \\
	H^j(B) =0 & j \neq 0,-1
\end{cases}. 
\]
Thus we obtain the following 
exact sequence of $\mca O_{X_0}$-module. 
\[
\xymatrix{
0  \ar[r] & H^{-1}(B) \ar[r]&  
H^0(A) \ar[r] & E \ar[r] & H^0(B) \ar[r] & H^{1}(A) \ar[r] & 0 
}
\]

It is enough to show that $H^{-1}(B)=H^0(A)=0$. 
Since the cohomology sheaves of 
$A$ and $B$ are $\mca O_{X_0}$ modules, 
they decompose as 
$H^i(A) = F^i _A \+ T^i_A$ and $H^i(B)=F^i_B \+ T^i_B$ 
where $F_A^i$ and  $F_B^i$  are locally free 
on $X_0$ 
and $T_A^i$ and $T_B^i$ are torsion sheaves on $X_0$. 

By the vanishing $\Hom_{\mca X}(H^0(A), H^{-1}(B))=0$, 
we have $\Hom_{X_0} (T^0_A, T^{-1}_B)=0$. 
Since $T_B^{-1}$ is a subobject of $T_A^0$, 
it must be zero, 
in particular $H^{-1}(B)$ is locally free.

Moreover the vanishing
$\Hom_{\mca X}^1(H^{0}(A), H^{-1}(B))=0$ implies 
$\Hom_{\mca X}^1(F^0_A, F^{-1}_B)=0$. 
Since $\Hom_{X_0}^1
(F^0_A, F^{-1}_B) \subset 
\Hom_{\mca X}^1(F^0_A, F^{-1}_B)
$, 
we see $\Hom_{X_0}^1(F^0_A, F^{-1}_B)=0$. 
By the Serre duality as $\mca O_{X_0}$ module, 
we have 
\begin{equation} \label{eq:1029}
0 = \Hom_{X_0}^1 (F^{0}_A, F^{-1}_B)  \cong 
\Hom_{X_0}^0 (F^{-1}_B, F^0_A \otimes \omega_{X_0})^{\vee}. 
\end{equation}
Since the genus of $X_0$ is non-zero, 
there is an inclusion 
\begin{equation}\label{eq:}
\Hom_{X_0}^0 (F^{-1}_B, F^0_A)
\subset 
\Hom_{X_0}^0 (F^{-1}_B, F^0_A \otimes \omega_{X_0}).  
\end{equation}
Thus (\ref{eq:1029}) implies 
$\Hom_{X_0}^0 (F^{-1}_B, F^0_A) =0$. 
Since $H^{-1}(B)$ is locally free, 
$F_B^{-1}$ is a subobject of $F_A^0$. 
Thus we have $H^{-1}(B)=0$.

We finally show that $H^1(A)=0$. 
Similarly to the case of $H^{-1}(B)$, 
the vanishing 
\[
\Hom_{\mca X}^1 (H^1(A), H^0(B))=0
\]
implies that 
$\Hom_{X_0}^1(F^1_A, F^0_B)=0$. 
Then, by 
Serre duality for $\mca O_{X_0}$ modules, 
we obtain 
\begin{equation}\label{eq:1029-l}
\Hom_{X_0}(F^0_B, F^1_A)=0. 
\end{equation}
If $F^1_A \neq 0$, then clearly we see $H^{1}(A) \neq 0$ and 
$\Hom(H^0(B), H^1(A))\neq 0$. 
This contradicts with (\ref{eq:1029-l}). 
Thus $F_A^1$ is zero and 
$H^1(A)$ is a torsion sheaf on $X_0$.

Moreover the vanishing 
$\Hom_{\mca X}^1(T^1_A, F^0_B)=0$ implies that 
$\Hom_{X_0}^1(T^1_A, F^0_B)=0$. 
Since $F_B^0$ is locally free on $X_0$, 
we obtain 
$T^1_A=0$ or $F^0_B=0$.  

Suppose to the contrary $T_A^1 \neq 0$. 
Then $F_B^0$ is zero, 
and we obtain a surjection 
$T^0_B \to T^1_A$. 
Thus the support of $T_A^1$ is contained in 
that of $T_B^0$. 
By (\ref{eq:1029-13}), we see 
$\Hom_{\mca X}(T^1_A, T^0_B)=0$
and this contradicts with $T_A^1 \neq 0$. 
Thus $T_A^1 = 0$ and, hence $H^1(A) = 0$ as well. 
\end{proof}

\begin{rmk}
Lemma \ref{lem:GKR} is a generalization of 
\cite[Lemma 7.2]{MR2084563}
to the case of the Artin base $\Spec R$. 
\end{rmk}

\begin{prop}\label{prop:macri}
Assume that $g(X_0)>0$. 
Then, 
for any $\sigma \in \Stab{\mb D^b (\mca X)}$, 
the structure sheaf $\mca O_x$ of any closed point $x \in X_0$ and 
any line bundle $L$ on $X_0$ are $\sigma$-stable. 
\end{prop}

\begin{proof}
Lemma \ref{lem:GKR} directly 
implies the semistability of $\mca O_x$ and of $L$.

If $L$ is not stable, 
we obtain the short exact sequence 
as $\mca O_{X_0}$-modules
\[
\xymatrix{
	0 \ar[r] & A \ar[r] & L  \ar[r] & B \ar[r] & 0, 
}
\]
with 
\begin{equation}\label{eq:1029samui}
\Hom_{\mca X}^{\leq 0}(A, B). 
\end{equation}
Applying Lemma \ref{lem:GKR}, 
we see that  $A$ is a line bundle on $X_0$ and 
$B$ is a torsion sheaf on $X_0$. 
The vanishing (\ref{eq:1029samui}) implies 
$A =0$ or $B=0$ which gives the proof. 
The proof for the stability of $\mca O_{x}$ is similar. 
\end{proof}

\begin{lem}\label{lem:injective}
Assume that $g(X_0)>0$. 
	Then the natural map 
	 $j_* {}^{-1} \colon \Dom{j_*} \to \Stab{\mb D^b (X_0)}$
	 is injective. 
\end{lem}

\begin{proof}
Since the functor 
\[
j_* \colon \mb D^b(X_0) \to \mb D^b(\mca X)
\] 
satisfies the condition (Ind), 
	Proposition \ref{prop:injectivity}, 
implies the desired assertion. 
\end{proof}

\begin{thm}\label{thm:main1}
Assume that $g(X_0)>0$. 
There exists a natural isomorphism 
\[
\Stab{\mb D^b(\mca X)} \cong \Stab{\mb D^b(X_0)}. 
\]
In particular, 
we have $\Stab{\mb D^b(\mca X)} \cong \widetilde{\mr{GL}} ^+_2(\bb R)$. 
\end{thm}

\begin{proof}
Let $j \colon X_0 \to X$ denote the closed embedding. 
	By Proposition \ref{prop:macri}, 
any stability condition on $\mb D^b(\mca X)$ satisfies 
the assumption of Proposition \ref{prop:key-1}. 
Thus the space $\Stab{\mb D^b(\mca X)}$ is contained in 
the orbit $\widetilde{\mr{GL}}_2^+(\bb R)\cdot \sigma_{\mr{st}}$. 
Hence $\Stab{\mb D^b(\mca X)} = \widetilde{\mr{GL}}_2^+(\bb R)\cdot \sigma_{\mr{st}}$ 
and $\Stab{\mb D^b(\mca X)} $ is connected.

Since $\sigma_{\mr{st}}$ is in $\Dom{j_*}$, 
we have 
$\Stab{\mb D^b (\mca X)} = \Dom{j_*}$. 
Thus we obtain a map of spaces 
\[
j_* ^{-1} \colon \Stab{\mb D^b(\mca X)} \to \Stab{\mb D^b (X_0)}. 
\]
Since $j_*^{-1} \sigma_{\mr{st}}$ gives a standard stability 
condition on $\mb D^b (X_0)$, 
the above map is surjective and is injective by Proposition \ref{prop:injectivity}. 
Thus the map is an isomorphism since both are complex manifolds. 
The last assertion follows from 
\cite[Theorem 2.7]{MR2335991}. 
\end{proof}

\section{The case of genus zero}

We consider the case of genus $0$. 
Since the projective line over a field is rigid, 
any deformation is trivial. 
Thus, we denote by $\bb P^1$ or by $\bb P^1_{\mb k}$ the 
projective line over the field $\mb k$, 
and by $\bb P^1 _{R}$ the projective line 
over $(R, \mf m)$.  
The closed immersion $\bb P^1 \hookrightarrow \bb P^1_R$ is denoted by 
$j \colon \bb P^1 \to \bb P^1 _{R}$. 
To study $\Stab{\mb D^b(\bb P^1_R)}$, 
we generalize the work of
Okada~\cite{MR2219846} and of Gorodentsev, Kuleshov, and Rudakov~\cite{MR2084563}.

We note that there is the following exact triangles in $\mb D^b(\bb P^1_R)$: 
\begin{align}
\label{eq:1123-1}
& \xymatrix{
j_* \mca O(k) \ar[r] & j_* \mca O_x \ar[r] & j_*\mca O(k-1)[1] 
}  (x \in \bb P^1)
 \\
\label{eq:1123-2}
& 
\xymatrix{
j_*\mca O(k)^{\+ n-k+1} \ar[r] & j_* \mca O(n) \ar[r] & j_* \mca O(k-1)[1]^{\+ n-k}
} (k<n) \\
\label{eq:1123-3}
& 
\xymatrix{
j_* \mca O(k)^{\+ k-n-1}[-1] \ar[r] & j_* \mca O(n) \ar[r] & j_* \mca O(k-1) ^{\+ k-n} 
} (n <k-1)
\end{align}
If $R=\mb k$, then the above sequence are the 
same as that of \cite{MR2084563}.

\begin{lem}\label{lem:0911JR}
Assume that non-zero coherent sheaves 
$E$ and $F$ on $\bb P^1$ satisfy	
\[
	\bb R\Hom_{\bb P^1_R} (j_*  E, j_* F) =0 \text{ and }
\Hom_{\bb P^1_R}^0 (j_*F, j_* E) \neq 0	. 
	\]
Then  the following holds 
	\[
	(E, F)= (\mca O_{\bb P ^1}(k)^{\+ m}, \mca O_{\bb P^1}(k-1)^{\+ n })  (n,m \in \bb N)
	 \]
\end{lem}

\begin{proof}
By the assumption, we have 
\begin{equation}\label{eq:1029-JR}
\Hom_{\bb P^1}^{0}(E, F) = \Hom_{\bb P^1}^1(E, F )=0
\end{equation}
Any sheaf on $\bb P^1$ decomposes 
into a direct sum of torsion sheaves and 
invertible shaef on $\bb P^1$. 
By the vanishing (\ref{eq:1029-JR}), 
$E$ and $F$ are both torsion sheaves on $\bb P^1$ or 
are the direct sums of invertible sheaves on $\bb P^1$. 

If they are torsion sheaves satisfying (\ref{eq:1029-JR}), 
the supports does not intersect each other. 
Hence $\Hom _{\bb P^1_R}(j_* F, j_*E)$ must be zero and 
this contradicts the assumption. 

Thus suppose that $E$ and $F$ are 
direct sums of invertible sheaves on $\bb P^1$. 
Since the vanishing of cohomologies 
\[
H^0 (\bb P^1, \mca O(k))=H^1(\bb P^1, \mca O(k))=0 \iff k=-1
\]
holds, the pair $(E, F)$ must be 
$
(\mca O_{\bb P ^1}(k)^{\+ m}, \mca O_{\bb P^1}(k-1)^{\+ n }) 
$. 
\end{proof}

\begin{lem}\label{lem:0917knym}
Assume that an exact triangle in $\mb D^b(\bb P^1_R)$ 
\[
\xymatrix{
	A \ar[r] & E \ar[r] & B \ar[r] & A[1], 
}
\] 
of non-zero objects 
satisfies 
$\Hom^{\leq 0}(A, B)=0$.

\begin{enumerate}
	\item Let $T_x$ be a torsion shaves on $\bb P^1$ supported on 
	a closed point $x \in \bb P^1$. 
If $E = j_* T_x$, then 
there exist $k \in \bb Z$ and $r \in \bb N $  such that 
	\[
	(A, B) \cong 
	(\mca O(k)^{\+ r}, \mca O(k-1)^{\+ r}[1])	. 
	\]
	\item Let $L$ be an invertible sheaf on $\bb P^1$. 
	If $E= j_ * L$, then 
there exist $k \in \bb Z$ and $r \in \bb N $ such that 
	\[
	(A, B) \cong (\mca O(k)^{\+ r}, \mca O(k-1)^{\+ r-1}[1])	
	\text{ or }
	(A, B) \cong (\mca O(k)^{\+ r}[-1], \mca O(k-1)^{\+ r})	
	\]
\end{enumerate}
\end{lem}

\begin{proof}
By Proposition \ref{prop:almost-hereditary}, 
$\mr{Coh}(\bb P^1)$ is almost hereditary in $\mr{Coh}(\bb P^1_R)$. 
In both cases, 
$E$ is an $\mca O_{\bb P^1}$ module. 
Hence similarly to the proof of 
Lemma \ref{lem:GKR}, 
each cohomologies of $A$ and $B$ are 
$\mca O_{\bb P^1}$-modules. 
As in the proof of Lemma \ref{lem:GKR}, 
we obtain the exact sequence of $\mca O_{\bb P^1}$-modules: 
\[
\xymatrix{
0 \ar[r] & H^{-1}(B) \ar[r] &  H^0(A) \ar[r] 
		& E \ar[r] & H^0(B) \ar[r] & H^1(A) \ar[r]  & 0. 
} 
\]
By Lemma \ref{lem:vanishing-for-HN}, 
we have 
\begin{align}
\label{eq:1029-18}
& \bb R\Hom_{  \bb P^1_R } (H^i(A), H^{i-1}(B)) = 0  \quad (i \in \bb Z), \\
& \Hom_{\bb P^1_R}^0( H^i(A), H^i(B)  ) =0 \quad (i \in \bb Z).   \label{eq:1029-19}
\end{align}

We show the first assertion. 
If $H^1(A) \neq 0$, 
$H^0(B)$ does not vanish. 
Applying Lemma \ref{lem:0911JR} to 
the pair $(H^1(A), H^0(B))$, 
they must be the pair of locally free sheaves on $\bb P^1$. 
Since $E$ is torsion on $\bb P^1$, 
they must be isomorphic and this contradicts with (\ref{eq:1029-18}). 
Hence $H^1(A)$ must be zero.

Next we show that 
$H^{-1}(B) \neq 0$. 
Suppose to the contrary that 
$H^{-1}(B)=0$, then the following 
sequence is exact as $\mca O_{\bb P^1}$-modules. 
\begin{equation}\label{eq:1029-20}
\xymatrix{
0 \ar[r] & H^0(A) \ar[r] & E \ar[r] & H^0(B) \ar[r] & 0
}
\end{equation}
Since $E$ is torsion on $\bb P^1$, 
so as $H^0(A)$ and $H^0(B)$. 
Since the support of $E$ is one point, 
we have 
$\Supp H^0(A) \cap \Supp H^0(B) \neq \emptyset$. 
This contradicts with 
$\Hom^{\leq 0}(A, B)=0$. 
Hence we have $H^{-1}(B) \neq 0$.

Then 
$H^0(A) \neq 0$ by the assumption. 
Applying 
Lemma \ref{lem:0911JR} to the pair 
$(H^0(A), H^{-1}(B))$, we  have
\begin{equation}\label{eq:tukaimawasi}
(H^0(A), H^{-1}(B)) = 
	(\mca O(k)^{\+ r}, \mca O(k-1)^{\+ r}[1])	. 
\end{equation}
If 
$H^0(B)$ is non-zero, then it is torsion on $\bb P^1$ 
and we have $\Hom_{\bb P^1}(H^0(A), H^0(B)) \neq 0$. 
This contradicts with (\ref{eq:1029-19}). 
Hence we have $H^0(B)=0$.

We prove the second assertion. 
Assume that $H^1(A)\neq 0$. 
Then, 
similarly to the first case, 
we have 
the same equality as in (\ref{eq:tukaimawasi}). 
Since $L$ is invertible sheaf on $\bb P^1$, 
the morphism from $L$ is injective.
Hence $H^{-1}(B)$ must be isomorphic to $H^0(A)$. 
Then the vanishing (\ref{eq:1029-18}) implies 
$H^{-1}(B)=H^0(A)=0$. 
Thus we 
obtain the following exact sequence of 
$\mca O_{\bb P^1}$-module. 
\[
\xymatrix{
0 \ar[r] & L  \ar[r] & \mca O(k-1)^{\+ r} \ar[r] & \mca O(k)^{\+ r-1} \ar[r] & 0
}
\]
This gives a desired assertion.

To complete the proof, assume that 
$H^1(A)=0$. 
If $H^{-1}(B)=0$, the we obtain the  
same exact sequence as in (\ref{eq:1029-20}) with $E= j_*L$. 
Then $H^0(A)$ is an invertible sheaf on $\bb P^1$ and 
$H^0(A)$ must be torsion on $\bb P^1$. 
This contradicts with (\ref{eq:1029-19}). 
Hence $H^{-1}(B)$ must be non-zero and 
this implies $H^0(A) \neq 0$. 

Then, by Lemma \ref{lem:0911JR}, 
we have 
\begin{equation}\label{eq:1029-22}
	(H^0(A), H^{-1}(B))= 
	(\mca O_{\bb P ^1}(k)^{\+ m}, \mca O_{\bb P^1}(k-1)^{\+ n }).
\end{equation}
Moreover, if $H^0(B)$ is non-zero, it is torsion on 
$\bb P^1$. 
Since $H^0(A)$ is locally free on $\bb P^1$ by 
(\ref{eq:1029-22}), 
we have 
\[
\Hom_{\bb P^1_R}( H^0(A), H^0(B)) \neq 0. 
\]
This contradicts with (\ref{eq:1029-19}). 
Hence $H^0(B)$ must be zero and this 
implies the desired assertion. 
\end{proof}


If the base ring is a field $\mb k$, 
the following lemma was essentially proved 
in the proof of  \cite[Theorem 6.2]{MR2084563}.

\begin{lem}\label{lem:GKR-key}
Assume that we have an exact triangle in $\mb D^b(\bb P^1_R)$
\[
\xymatrix{
	A \ar[r] & j_* \mca O(k-1)[1] \ar[r] & B 
}
\]
satisfying $\Hom^{\leq 0}(A, B)=0$. 
If moreover $\Hom ^{0}(\mca O(k), A)=0$, then $A=0$. 
\end{lem}

\begin{proof}
Suppose $A \neq 0$. 
	By Lemma \ref{lem:0917knym}, 
$A$ is a direct sum of objects form
$j_* \mca O(\ell)[1]$ with $\ell < k-1$ or 
$j_* \mca O(\ell)$ with $\ell > k$. 
If $\ell < k-1$, then $\Hom^0 (j_* \mca O(k), j_* \mca O(\ell)[1])\neq 0$. 
If $\ell >k$, then $\Hom^0(j_* \mca O(k), j_* \mca O(\ell))\neq 0$. 
Hence $A$ must be zero. 
\end{proof}

In the study of stability conditions on algebraic curves, 
the stability of skyscraper sheaves and that of line bundles 
plays a fundamental role. 
In \cite{MR2219846}, 
Okada studies the semistability of skyscraper sheaves
and of line bundles. 
The following proposition is a generalization 
of the study to the case of an Artinian ring.

\begin{lem}\label{lem:okada-san}
Let $\sigma$ be a locally finite stability condition on $\mb D^b(\bb P^1_R)$. 
Then the following holds: 
\begin{enumerate}
	\item There exists a non-$\sigma$-semistable line bundle on $\bb P^1$ (not on $\bb P^1_R$)
	if and only if there exists an integer $k$ such that 
	$j_* \mca O(k)$ and $j_* \mca O(k-1)[1]$ are $\sigma$-
	semistable with 
	$\phi (j_* \mca O(k)) > \phi (j_* \mca O(k-1)[1])$. 
	\item Let $x \in \bb P^1$. 
	Any skyscraper $\mca O_x$ on $\bb P^1$ is $\sigma$-semistable if and only if 
	any line bundle on $\bb P^1$ is $\sigma$-semistable. 
	\item Assume that any line bundle $j_* \mca O(k)$ is $\sigma$-semistable with 
	the phase $\phi_k$. 
	Then the inequalities 
	\begin{equation}\label{eq:okada-san}
	\phi _{k-1} \leq \phi_k \leq \phi_{k-1}+1	
	\end{equation} 
	hold for any $k$. 
\end{enumerate}
\end{lem}

\begin{proof}
We prove the assertion (1). 
The ``if'' part follows from the exact triangles 
(\ref{eq:1123-2}) or (\ref{eq:1123-3}). 

Suppose that there exists $\mca O(n) \in \bb P^1$ such that 
$j_* \mca O(n)$ is not semistable. 
Taking the maximal semistable factor of $j_* \mca O(n)$, 
we obtain the exact triangle 
\[
\xymatrix{
	A_1 \ar[r] & j_* \mca O(n) \ar[r] & A_2 \ar[r] & A_1[1], 
}
\]
with $\Hom ^{\leq 0} (A_1, A_2)=0$ and $A_1$ is semistable. 
By Lemma \ref{lem:0917knym}, 
the pair $(A_1, A_2)$ is isomorphic to 
\[ 
(A_1, A_2) \cong 
\begin{cases}
(j_* \mca O(k)^{\+}, j_* \mca O(k-1)[1]^{\+ }) & (k<n) \\
(j_* \mca O(k)^{\+}[-1], j_* \mca O(k-1)^{\+ }) & (n < k-1) \\
\end{cases}
\]
If $A_2$ is not semistable, 
then it has the maximal semistable factor $A_3$. 
Thus we obtain the exact triangle 
\[
\xymatrix{
	A_3 \ar[r] & j_* \mca O(k-1)[1] \ar[r] & A_4 \ar[r] & A_3[1]. 
}
\]
Then Lemma \ref{lem:GKR-key} implies that $A_3 =0$ and 
hence $A_2$ is semistable. 
Thus $j_* \mca O(k)$ and $j_* \mca O(k-1) $ are semistable 
with $\phi (j_* \mca O(k)) > \phi (j_* \mca O(k-1)[1])$.

We prove the assertion (2). 
Assume that the exists a non-semistable line bundle $O(n)$ on $\bb P^1$, 
by the assertion (1), we obtain two line bundles on $\bb P^1$ such that 
$j_* \mca O(k)$ and $j_* \mca O(k-1)[1]$ are both semistable 
with $\phi (j_* \mca O(k)) > \phi (j_* \mca O(k-1)[1])$. 
Then the sequence (\ref{eq:1123-1}) gives the HN filtration of $j_* \mca O_x$ 
and $j_* \mca O_x$ is not semistable for any $x \in \bb P^1$. 

Assume that there exists a point $x \in \bb P^1$ such that 
$j_* \mca O_x$ is not semistable. 
Taking the maximal semistable factor of $j_* \mca O_x$, by Lemma \ref{lem:0917knym}, 
we obtain the exact triangle 
\[
\xymatrix{
j_* \mca O(k) \ar[r] & j_* \mca O_x \ar[r] & j_* \mca O(k-1)[1] 
}
\]
for some $k \in \bb Z$. 
In particular, $j_* \mca O(k)$ is semistable since it 
is the maximal semistable factor. 

If $j_* \mca O(k-1)[1]$ is not semistable, 
the maximal semistable factor $A$ of $j_* \mca O(k-1)[1]$ 
satisfies $\Hom ^{\leq 0}( j_* \mca O(k), A)=0$. 
Then Lemma \ref{lem:GKR-key} implies $A=0$ and hence 
$j_* \mca O(k-1)[1]$ is semistable. 
Thus the HN filtration of $j_* \mca O_x$ is 
given by the sequence (\ref{eq:1123-1}). 
Hence the phases satisfy 
$\phi (j_* \mca O(k)) > \phi (j_* \mca O(k-1)[1])$. 
Then the assertion (1) implies that 
there exists a non semistable line bundle on $\bb P^1$.

To prove the assertion (3), 
assume that $j_* \mca O(n)$ is semistable for any $n \in \bb Z$. 
Since $\Hom(J_* \mca O(k-1), j_* \mca O(k)) \neq 0$, 
the phases satisfy $\phi _{k-1} \leq \phi_k$ for any $k \in \bb Z$. 
Moreover, it the exists a $k \in \bb Z$ such that 
$\phi_{k} > \phi_{k-1}+1$, then 
the sequence (\ref{eq:1123-2}) gives 
the HN filtration of $j_* \mca O(n)$ for $n >k$ and 
$j_* \mca O(n)$ is not semistable. 
This contradicts the assumption. 
Hence we have 
$\phi_k \leq \phi _{k-1}+1$ for any $k \in \bb Z$.  
\end{proof}

To study the stability of skyscraper sheaves and of line bundles, 
we first record the following key observation.

\begin{lem}\label{lem:0213}
If an object $S \in \mb D^b(\bb P^1_R)$ satisfies 
\begin{enumerate}
	\item $\Hom_{\bb P^1_R}^{<0}(S, S)=0$, and 
	\item any non-zero endomorphism of $S$ is invertible, 
\end{enumerate}
then $S$ is isomorphic to 
the skyscraper sheaf $j_* \mca O_x$ or an invertible sheaf $j_* \mca O(k)$ 
on $\bb P^1$ up to shifts. 
\end{lem}

\begin{proof}
Since any element in the ideal $\mf m$ is nilpotent, 
the action of $\mf m$ on $S$ must be zero y
by the assumption (2). 
Thus 
any cohomology of $S$ with respect to the standard 
$t$-structure
is in $\mr{Coh}(\bb P_1)$. 
Since $\mr{Coh}(\bb P^1)$ is almost hereditary in 
$\mr{Coh}(\bb P^1_R)$, 
Proposition \ref{prop:class-of-stable} implies that 
$S$ is in $\mr{Coh}(\bb P^1)$ up to shifts.

Moreover, $S$ is indecomposable by the assumption (2) by 
\cite[Lemma 3.3]{MR4517994}. 
Then $S$ must be a torsion sheaf $j_*T_x$ 
supported on a point $x \in \bb P^1$ or 
an invertible sheaf $j_* \mca O(n)$ on $\bb P^1$. 
Assume that the torsion sheaf is not 
isomorphic to the skyscraper sheaf $\mca O_x$. 
Then there exists a natural quotient $T_x \to \mb k$. 
Since the kernel of $T_x \to \mb k$ contains $\mb k$, 
we obtain a non-zero endomorphism ot $T_x$ as follows: 
\[
\varphi \colon T_x \to \mb k \subset T_x
\]
Then the morphism $\varphi$ is nilpotent, that is, $\varphi ^2=0$. 
This contradicts the assumption. 
Hence $T_x$ must be the skyscraper sheaf 
$j_* \mca O_x$. 
\end{proof}

\begin{lem}\label{lem:class-stable-P1}
If an object $A$ in $\mb D^b(\bb P^1_R)$ is stable 
with respect to a stability condition $\sigma$ on $\mb D^b (\bb P^1)$, 
then $A$ must be a skyscraper sheaf $\mca O_x$ or 
a line bundle on $\bb P^1_{\mb k}$ up to shifts.  
\end{lem}

\begin{proof}
If an object $A \in \mb D^b(\bb P^1_R)$ is stable, 
then $A$ satisfies the assumption of 
Lemma \ref{lem:0213}. 
Hence Lemma \ref{lem:0213} implies the desired assertion. 
\end{proof}

We further develop Okada's observation's on semistability by using 
Lemma \ref{lem:class-stable-P1}.

\begin{lem}\label{lem:1128TNJ}
Assume that any line bundle $\mca O(k)$ on $\bb P^1_{\mb k}$ is semistable 
with the phase $\phi _k$. 
For any point $x \in \bb P^1_{\mb k}$, 
the skyscraper sheaf $\mca O_x$ is not stable but semistable if and only if 
	there exists an integer $k$ such that $\phi _k = \phi _{k-1}+1$ 
\end{lem}

\begin{proof}
Recall the exact triangle in (\ref{eq:1123-1}) for any $k \in \bb Z$ and any 
$x \in \bb P^1$. 

If there exists $k \in \bb Z$ such that $\phi_k = \phi _{k-1}+1$, 
then the above triangle gives a Jordan-Holder 
filtration of $\mca O_x$ for any $x \in \bb P^1$. 
Hence $\mca O_x$ is not stable. 
The semistability of $\mca O_x$ follows from 
Lemma \ref{lem:okada-san}.

Conversely, assume that 
any skyscraper sheaf is not stable but semistable. 
Then there exists a stable object $A \in \mb D^b(\bb P^1_R)$ 
such that $\Hom_{\bb P^1}(A, j_* \mca O_x) \neq 0$ and 
\begin{equation}\label{eq:1121}
\phi(A) = \phi(j_*\mca O_x)	
\end{equation}
Then Lemma \ref{lem:class-stable-P1} implies that 
$A= j_* \mca O(k)$ for some $ k \in \bb Z$. 
Since the cone of $j_* \mca O(k) \to j_* \mca O_x$ is 
$j_* \mca O(k-1)[1]$, 
the equation (\ref{eq:1121}) implies $\phi_k = \phi _{k-1}+1$. 
\end{proof}

\begin{lem}\label{lem:1123-F1}
For a stability condition $\sigma \in \Stab{\mb D^b(\bb P^1_R)}$, 
assume that there exists an integer $k \in \bb Z$ such that 
$j_* \mca O(k)$ and $j_* \mca O(k-1)$ are semistable with phases 
$\phi _k$ and $\phi_{k-1}$, respectively. 
	If $\phi_k \geq  \phi_{k-1}+1$, then 
	$j_* \mca O(k)$ and $j_* \mca O(k-1)$ are stable. 
\end{lem}

\begin{proof}
Since our stability condition is locally finite, 
any object is given by a successive extension of 
stable objects. 
Thus the classes of stable objects generate the group $K_0(\bb P^1_R)$. 
Since $\rank K_0(\bb P^1_R)=2$, 
there must be at least two stable objects for any $\sigma$.

Suppose the inequality $\phi _k > \phi_{k-1}+1$ holds. 
Then 
if the degree of a line bundle $\mca O(n)$ on $\bb P^1$ is neither $k$ nor $k-1$, 
$j_* \mca O(n)$ is 
not semistable by the triangles (\ref{eq:1123-2}) and (\ref{eq:1123-3}). 
Moreover any skyscraper sheaf is not semistable by 
the triangle (\ref{eq:1123-1}). 
Thus there exactly two semistable objects $\mca O(k)$ and $\mca O(k-1)$. 
Since stable object is semistable, 
Lemma \ref{lem:class-stable-P1} implies the desired assertion.

Suppose $\phi_k = \phi_{k-1}+1$. 
Then skyscraper sheaf $j_* \mca O_x$ is not stable 
but semistable by Lemma \ref{lem:1128TNJ}. 
Moreover any line bundle of degree other than 
$k$ or $k-1$ is not stable but semistable 
by the triangles (\ref{eq:1123-2}) and (\ref{eq:1123-3}). 
Thus $j_* \mca O(k)$ and $j_* \mca O(k-1)$ must be stable by Lemma \ref{lem:class-stable-P1}. 
\end{proof}

\begin{rmk}
Assume that there is $k \in \bb Z$ such that 
$\mca O(k)$ and $\mca O(k-1)$ is stable with $\phi_{k} \geq \phi_{k-1}[1]$. 
Then such a $k$ is unique. 
In fact, if the equality $\phi_{k}>\phi_{k-1}+1$ holds, then 
any line bundle except $\mca O(k)$ and $\mca O(k-1)$ is not semistable. 
If the equality $\phi_{k}=\phi_{k-1}+1$ holds, 
then 
any line bundle except $\mca O(k)$ and $\mca O(k-1)$ is not stable. 
\end{rmk}

\begin{prop}\label{prop:1122}
Fix a stability condition $\sigma =(Z, \mca P)$ on $\mb D^b (\bb P^1_R)$. 
\begin{enumerate}
	\item If any skyscraper sheaf $j_* \mca O_x$ of a point $x \in \bb P^1_{\mb k}$ is stable, 
	then any line bundle $j_* L$ on $\bb P^1$ is stable. 
	\item Conversely, if any line bundle $L$ on $\bb P^1$ is stable, then any skyscraper sheaf 
	$j_* \mca O_x$ is stable. 
\end{enumerate}	
\end{prop}

\begin{proof}
We first prove the first assertion. 
By the assumption, Lemma \ref{lem:okada-san} 
implies that any line bundle $j_* L $ on $\bb P^1$ is semistable. 
If $j_* L$ is not stable, 
then the line bundle should be given by 
a successive extension of skyscraper sheaves 
since $\sigma$ is locally finite. 
This clearly gives a contradiction, 
hence there exists at least one line bundle $\mca O(k)$ 
such that $j_* \mca O(k)$ is stable.

Suppose that a line bundle $j_* L$ on $\bb P^1$ is not stable. 
Then there is a stable object $A$ such that 
$\Hom_{\bb P^1_R} (A, j_*L)\neq 0$ and $\phi(A)=\phi(j_* L)$. 
Then the value $Z(j_* L)$ is contained in a real line spanned by 
$Z(A)$ in $\bb C$: 
\[
Z(j_* L) \in \bb R \cdot Z(A). 
\]

By Lemma \ref{lem:class-stable-P1}, 
up to shift, 
$A$ is $j_* \mca O_x$ or $j_* \mca O(\ell)$ 
for some $x \in \bb P^1$ and $\ell \in \bb Z$. 
In any cases, 
note that the classes $[L]$ and $[A]$ give 
a basis of the vector space $K_0(\bb P^1_R) \otimes _{\bb Z} \bb R$. 
Hence the image $V_{\sigma}$ of $Z \colon K_0(\bb P^1) \to \bb C$ is contained 
in a one dimensional $\bb R$ vector space in $\bb C$.

On the other hand, 
recall that $j_* \mca O_x$ and $j_* \mca O(k)$ are stable.  
Let $\phi _x$ and $\phi_k$ be the phases of $j_*\mca O_x$ and $j_* \mca O(k)$ respectively. 
The non-vanishings 
\[
\Hom_{\bb P^1_R} ( j_* \mca O(k), j_* \mca O_x) \neq 0 
\text{ and } \Hom_{\bb P^1_R}(j_* \mca O_x[-1], j_* \mca O(k)) \neq 0
\]
implies the inequalities
\begin{equation}\label{eq:1121KRS}
	\phi_x -1 < \phi _k  < \phi _x. 
\end{equation}
Since the classes $[j_ * \mca O_x]$ and $[j_* \mca O(k)]$ give a basis 
of $K_0(\bb P^1)  \otimes _{\bb Z} \bb R$, 
the inequalities (\ref{eq:1121KRS}) implies that 
the image $V_{\sigma}$ does not contained 
in a real $1$ dimensional line in $\bb C$ and 
this gives a contradiction. 
Hence any line bundle must be stable.

To prove the second assertion, 
suppose that any line bundle $j_* \mca O(k)$ is stable with 
the phase $\phi _k$. 
Then any skyscraper sheaf $j_* \mca O_x$ is semistable by Lemma \ref{lem:okada-san}.

Suppose to the contrary that 
there exists a non-stable skyscraper sheaf $j_* \mca O_x$. 
By Lemma \ref{lem:1128TNJ}, 
there exists an integer $k_0 \in \bb Z$ such that 
$\phi _{k_0} = \phi _{k_0-1}+1$. 
Then we see that $j_* \mca O(k_0+1)$ is not stable but semistable by the 
exact triangle
\[
\xymatrix{
j_* \mca  O(k_0)^{\+2} \ar[r] & j_* \mca O(k_0+1) \ar[r] & j_* \mca O(k_0-1)[1]. 
}
\]
This contradicts the assumption. 
Hence any skyscraper sheaf is stable. 
\end{proof}

\begin{prop}\label{prop:1123gogo}
Let $\sigma$ be in $\Stab{\mb D^b(\bb P^1_R)}$. 
Then exactly one of the following holds. 
\begin{enumerate}
	\item There exists an integer $k$ such that 
	$h_* \mca O(k)$ and $j_* \mca O(k-1)$ are only stable with 
	$\phi (j_* \mca O(k))  \geq \phi ( j_* \mca O(k-1)[1])$. 
\item Any line bundle and any skyscraper sheaf on $\bb P^1_{\mb k}$ are stable. 
\end{enumerate}	
\end{prop}


\begin{proof}
For any $\sigma$, we note that 
exactly one of the following holds: 
\begin{itemize}
	\item[(a)] There exists a non-semistable line bundle on $\bb P^1$. 
	\item[(b)] Any line bundle on $\bb P^1$ is semistable. 
\end{itemize}
In case (a), 
there exists an integer $k$ such that 
$j_* \mca O(k)$ and $j_* \mca O(k-1)[1]$ are semistable and satisfy
$\phi (j_* \mca O(k)) > \phi (j_* \mca O(k-1)[1])$ by Lemma \ref{lem:okada-san}. 
Then Lemma \ref{lem:1123-F1} implies that 
$j_* \mca O(k)$ and $j_* \mca O(k-1)$ are stable.

In case (b), 
let $\phi _k$ be the phase of $j_* \mca O(k)$. 
By Lemma \ref{lem:okada-san}, 
the inequalities 
\begin{equation}\label{eq:1123}
\phi (j_* \mca O(k)) \leq  \phi ( j_* \mca O(k-1)[1])
\end{equation}
hold for any $k$. 
If 
the equation in (\ref{eq:1123}) holds for some $k \in \bb Z$, 
Lemma \ref{lem:1123-F1} implies that 
$j_* \mca O(k)$ and $j_* \mca O(k-1)$ are stable. 
Thus the (1) holds. 

Finally suppose that 
inequalities $\phi _k < \phi _{k-1}+1$ holds for any $k$. 
Then we see that any skyscraper sheaf is stable 
by Lemma \ref{lem:1128TNJ} and that 
any line bundle is stable by Proposition \ref{prop:1122}.  
This is the case (2). 
\end{proof}

From now on, 
the smallest triangulated subcategory of 
$\mb D^{b}(\bb P^{1}_{R})$ 
containing $j_* \mca O(k)$ for $k \in \bb Z$ 
is denoted by $\mb T(k)$: 
\[
\mb T(k) = \< j_{*} \mca O(k) \>. 
\]

\begin{lem}
The category $\mb T (k)  $ is equivalent to $\mb D^b(R)$. 
\end{lem}

\begin{proof}
	Consider two exact functors between 
$\mb D^b(\bb P^1_R)$ and $\mb D^b(R)$ as follows: 
\begin{equation*}
	\xymatrix{
 \mb D^b(R) \ar@<0.5ex>[r]^-{F} & \ar@<0.5ex>[l]^-{G}  \mb D^b (\bb P^1_R), 
	}
\end{equation*} 
where $F(-) = \mca O_{\bb P^1_R}(k) \otimes _R (-)$ and 
$G(-) =\bb R \Hom(\mca O_{\bb P^1_R}(k), -)  $. 
Then, for $M \in \mb D^b(R)$, we have 
\begin{align*}
G \circ F(M) &= \bb R\Hom(\mca O_{\bb P^1_R}(k), \mca O_{\bb P^1_R}(k)\otimes _R M) \\
&\cong \bb R\Hom(\mca O_{\bb P^1_R}(k), \mca O_{\bb P^1_R}(k)) \otimes _R M 	\cong M. 
\end{align*}
Since $F$ the left adjoint of $G$, $F$ is fully faithful. 
Moreover, one easily see $F(R/\mf m)= j_* \mca O_{\bb P^1}(k)$. 
Hence the image of $F$ is contained in 
$\mb T(k)$ and essentially surjective to it's image. 
Thus we have proved the assertion. 
\end{proof}

\begin{lem}\label{lem:1118}
Any semi-orthogonal decomposition of $\mb D^b(\bb P^1)$ can 
extends to $\mb D^b(\bb P^1_R)$. 
The pair $(\mb T(k-1), \mb T(k))$ gives a semi-orthogonal decomposition 
of $\mb D^b(\bb P^1_R)$. 
\end{lem}

\begin{proof}
By Lemma \ref{lem:260308} 
	we see that 
	the pair $(\mb T(k-1), \mb T(k))$ is semi-orthogonal. 
Thus it suffices to show 
that the pair 
generates $\mb D^b(\bb P^1_R)$. 
Since $\mca O(k)$ and $\mca O(k-1)$ generate 
the category $\mb D^b(\bb P^1)$, 
the full-sub abelian category $\mr{Coh}(\bb P^1)$ of $\mr{Coh}(\bb P^1_R)$  
is contained in 
the extension closure $\< j_* \mca O(k-1), j_* \mca O(k)  \>$ of 
$j_* \mca O(k-1)$ and $j_* \mca O(k)$. 
Since the extension closure $\< j_* \mr{Coh}(\bb P^1) \>$ is 
the whole category $\mb D^b(\bb P^1_R)$, 
the pair $(\mb T(k-1), \mb T(k))$ generates $\mb D^b(\bb P^1_R)$. 
\end{proof}


\begin{prop}\label{prop:gluing-k}
	Take a stability condition $\sigma =(Z, \mca P) \in \Stab{\mb D^b(\bb P^1_R)}$. 
	Assume that $j_* \mca O(k)$ and $j_* \mca O(k-1)$ are stable with 
	phases $\phi(j_* \mca O(k)) \geq  \phi(j_* \mca O(k-1)[1])$. 
Then there exist stability conditions $\tau_1 \in \Stab{\mb T(k-1) }$
 and $\tau_2 \in \Stab{\mb T(k) }$ 
 such that $\sigma = \gl{\tau_1}{\tau_2}$. 
\end{prop}

\begin{proof}
Let $j_n \colon \mb T(n) \to \mb D^b(\bb P^1_R)$ be the embedding functor. 
	Since $\mb T(n)$ is equivalent to $\mb D^b (R)$, 
for any $\Stab{\mb T(n)}$, $j_* \mca O(n)$ is stable. 
Hence the stability condition $\sigma$ is in $\Dom{j_k} \cap \Dom{j_{k-1}}$. 
by the assumption.  

Put $\tau _i = j_i^{-1}\sigma$ ($i=k-1, k$) which gives a stability condition 
on $\mb T(i)$. 
Recall that the pair $(\mb T(k-1), \mb T(k))$ gives an semiorthogonal 
decomposition on $\mb D^b(\bb P^1_R)$. 
We wish to show that $\sigma = \gl{\tau_{k-1}}{\tau_k}$. 
Since the central charge of $\sigma$ is clearly glued from 
those of $\tau_{k-1}$ and of $\tau_{k}$, 
it is enough to show that the heart of $\sigma$ is glued from 
hearts of $\tau_{k-1}$ and of $\tau_{k}$. 

Choose an integer $m$ satisfying 
$m < \phi(j_* \mca O(k-1)[1]) \leq m+1 $. 
By the assumption for phases, 
there exists a non-negative integer $ \ell $ satisfying 
$m < \phi (j_* \mca O(k))-\ell \leq m+1$. 

Since $\mb T(n)$ is equivalent to $\mb D^b(\Spec\, R)$, 
both objects 
$j_* \mca O(k-1)[1-m]$ and $j_* \mca O(k)[-m-\ell] $ determine 
the hearts $\mca A_{k-1}$ and $\mca A_{k}$ of bounded 
$t$-structures on $\mb T(k-1)$ and on $\mb T(k)$. 
Moreover 
we clearly see 
\[
\Hom^{\leq 0} (j_* \mca O(k-1)[1-m], j_* \mca O(k)[-m -\ell]) 
= \Hom^{\leq 0} (j_* \mca O(k-1), j_* \mca O(k)[-1 -\ell]) =0. 
\]
Thus, 
we obtain the gluing heart $\gl{\mca A_{k-1}}{\mca A_k}$. 
Since 
$j_* \mca O(k-1)[1-m]$ and $j_* \mca O(k)[-m-\ell] $ are in 
$\mca P_{\sigma}(0,1]$, 
the gluing heart coincides with $\mca P_{\sigma}(0,1]$. 
\end{proof}

\begin{lem}\label{lem:onaji-P^1}
For the functor 
	$j_* \colon \mb D^b(\bb P^1) \to \mb D^b(\bb P^1_{R})$, 
	we have
\[
\Dom{j_*} = \Stab{\mb D^b(\bb P^1_R)}. 
\]
\end{lem}

\begin{proof}
If $\sigma \in \Stab{\mb D^b(\bb P^1_R)}$ is given, 
the HN filtration of a direct sum is the direct sum 
of the HN filtrations by Lemma \ref{lem:direct-sum}. 
Since any object in $\Im j_*$ is a 
direct sum of shifts of sheaves on $\bb P^1$, 
it is enough to show that 
each $\sigma$-semistable factor of 
any sheaf on $\bb P^1$ is in $\Im j*$. 

This essentially follows from Lemma \ref{lem:0917knym}. 
If a sheaf on $\bb P^1$ is indecomposable, 
then it must be a torsion sheaf $T_x$
supported on a point $x \in \bb P^1$ or 
an invertible sheaf $L$ on $\bb P^1$. 

If $T_x$ is semistable, then nothing to prove. 
If not, by Proposition \ref{prop:1123gogo}, 
there exists $k \in \bb Z$ such that 
$j_* \mca O(k)$ and $j_* \mca O(k-1)[1]$ are stable with 
$\phi (j_* \mca O(k)) \geq \phi (j_* \mca O(k-1)[1])$. 
Hence the HN filtration of $j_* T_x$ is given by the form 
\[
\xymatrix{
j_* \mca O(k)^{\+ r} \ar[r]  & j_* T_x \ar[d] \\
& j_* \mca O(k-1)^{\+ r'}[1]\ar@{-->}[ul]
}
\]
Hence any semistable factor of $T_x$ is in the image $\Im j_*$.

Suppose that $E = j_* \mca L$, 
where $\mca L$ is a line bundle on $\bb P^1$. 
If $j_* \mca L$ is not semistable, 
Proposition \ref{prop:1123gogo} implies that 
there exists an integer $k \in \bb Z$ such that 
$j_* \mca O(k)$ and $j_* \mca O(k-1)[1]$ are stable and 
satisfy 
$\phi (j_* \mca O(k)) \geq \phi (j_* \mca O(k-1)[1])$. 
Hence the HN filtration of $j_* \mca L$ is of the form 
(\ref{eq:1123-2}) or (\ref{eq:1123-3}) and 
the semistable factors of $j_* \mca L$ all belong to 
$\Im j_*$. 
Therefore, we conclude that $\sigma \in \Dom{j_*}$. 
\end{proof}

\begin{lem}\label{lem:connected}
The space $\Stab{\mb D^b (\bb P^1_R)}$ is connected. 	
\end{lem}

\begin{proof}
Let $\sigma \in \Stab{\mb D^b (\bb P^1_R)}$. 
If $\sigma$ satisfies the condition (2) of Proposition \ref{prop:1123gogo}, 
then $\sigma \in  \widetilde{\mr{GL}}_2^+ (\bb R) \cdot \sigma_{\mr{st}}$ by 
Proposition \ref{prop:key-1}.

Assume that $\sigma $ satisfies the condition (1)of Proposition \ref{prop:1123gogo}. 
Then there exists an integer $k \in \mb Z$ such that 
$j_* \mca O(k)$ and $j_* \mca O(k-1)$ are stable with 
$\phi (j_* \mca O(k)) \geq \phi (j_* \mca O(k-1)[1])$. 
By Proposition \ref{prop:gluing-k}, 
we can write 
$\sigma =\gl{\tau_{k-1}}{\tau_k}$ with 
$\tau _i \in \Stab{\mb T(i)}$.  
Using the action of $\widetilde{\mr{GL}}_2^+ (\bb R)$, 
deform $\tau_{k}$ to $\tilde{\tau}_{k}$ so that 
$j_* \mca O(k)$ and $j_* \mca O(k-1)[1]$ are stable with 
\[
\phi (j_* \mca O(k)) < \phi (j_* \mca O(k-1)[1]), 
\] 
for $\tilde{\sigma}=\gl{\tau_{k-1}}{\tilde{\tau}_k}$. 
Then $\tilde{\sigma}$ 
satisfies the condition (2) in Proposition \ref{prop:1123gogo}. 
Indeed, if this was not the case, 
then $j_* \mca O(k-1)$ and $j_{*}\mca O(k)$ 
would fail to be semistable
contradicting the assumption.  
Hence $\tilde{\sigma}$ lies in the connected component
$\widetilde{\mr{GL}}_2^+ (\bb R) \cdot \sigma_{\mr{st}}$. 
\end{proof}

\begin{rmk}\label{rmk:related-work}
If the Artinian local ring $R$ is the field $\mb k$, 
then the embedding $\bb P^1 \to \bb P^1_R$ is an isomorphism. 
Hence the results obtained so far can be regarded 
as an extension of Okada's results \cite{MR2219846} to Artin local rings.
\end{rmk}

\begin{thm}\label{thm:main2}
For the functor $j_* \colon  \mb D^b(\bb P^1) \to \mb D^b(\bb P^1_R)$, 
the map
\[
j_* ^{-1} \colon \Stab{\mb D^b(\bb P^1_R)} \to \Stab{\mb D^b(\bb P^1)}
\]
gives an isomorphism of complex manifolds. 
\end{thm}

\begin{proof}
By Lemma \ref{lem:onaji-P^1}, 
we have $\Dom{j_*}= \Stab{\mb D^b(\bb P^1_R)}$. 
Moreover, $\Stab{\mb D^{b}(\bb P^{1}_{R} ) }$ is connected by Lemma \ref{lem:connected}. 
To complete the proof, it suffices to show that 
the map  $j_{*}^{-1}$ is bijective 
since both $\Stab{\mb D^{b} ( \bb P^1 )    }$ and 
$\Stab{\mb D^b(\bb P^1_R)}$ are connected complex 
manifolds. 
Injectivity has already been established in Lemma~\ref{lem:injective},
so it remains to prove the surjectivity.

Let $\sigma \in \Stab{\mb D^b(\bb P^1)}$. 
By Proposition \ref{prop:key-1} applied to the case $R=\mb k$, 
it is enough to discuss the following two cases: 
\begin{enumerate}
	\item any line bundle and any skyscraper sheaf is stable. 
	\item there is an integer $k \in \bb Z$ such that 
	$\mca O(k)$ and $\mca O(k-1)[1]$ are stable and satisfy $\phi(\mca O(k)) \geq \phi (\mca O(k-1)[1] )$
\end{enumerate}
In both cases, we show that the stability condition 
$\sigma$ is in the image $\Im j_*^{-1}$.

Suppose that case (1) holds. 
Applying Proposition \ref{prop:key-1} to the case $R=\mb k$, 
we may assume that $\sigma$ is 
the standard stability condition on $\mb D^b(\bb P^1)$ up to the $
\widetilde{\mr{GL}}_2^+(\bb R) $-action. 
Since the standard stability condition 
$\sigma_{\mr{st}}$ on $\mb D^b (\bb P^1_R)$ maps to 
the stability condition $\sigma$ via the map $j_*^{-1}$. 
Hence $\sigma $ is in $\Im j_*^{-1}$.

Suppose that case (2) holds. 
By Proposition \ref{prop:gluing-k} to the case $R= \mb k$, 
there exist $\tau_1 \in \Stab{ \< \mca O(k-1) \>}$ and 
$\tau_2 \in \Stab{\< \mca O(k) \>}$
such that $\sigma =\gl{\tau_1}{\tau_2}$. 
Note that the functor $j_* \colon \mb D^b (\bb P^1) \to \mb D^b (\bb P^1_R)$ 
sends $\< \mca O(n) \>$ to $\< j_* \mca O(n) \> = \mb T(n)$ for any $n \in \bb Z$. 
Since the category $\mb T(n)$ is equivalent to $\mb D^b(R)$, 
we have $\tilde {\tau}_1 \in \Stab{\mb T (k-1) }$ and 
$\tilde {\tau }_2 \in \Stab{\mb T(k) }$ such that 
$j_* ^{-1} \tilde{\tau} _i = \tau _i$ ($i=1,2$) by Proposition \ref{prop:key-prop}.

Then we can glue $\tilde{\tau}_1$ and $\tilde{\tau}_2$ as follows. 
Put $\tilde {\tau}_i = (Z_i, \mca A_i)$ and 
the slicing of $\tilde{\tau}_i$ is denoted by $\mca P_i$ ($i=1,2$). 
Indeed, 
$j_* \mca O(k)$ and $j_* \mca O(k-1)$ are stable with respect to 
$\tilde {\tau}_1$ and to $\tilde{\tau}_2$. 
Let $\phi_k$ and $\phi_{k-1}$ be the phase of 
$j_* \mca O(k)$ and of $j_* \mca O(k-1) $. 
By Proposition \ref{prop:gluing-k}, 
we obtain the gluing stability condition $\tilde{\sigma}=
\gl{\tilde{\tau_1}}{\tilde{\tau_2}}$. 
Then the stability condition $\tau = \gl{\tilde{\tau}_1}{\tilde{\tau}_2}$ 
satisfies 
\[
j_*^{-1} \tau = 
j_*^{-1} 
\gl{\tilde{\tau}_1}{\tilde{\tau}_2} = 
\gl{j_*^{-1} \tilde{\tau}_1}{j_*^{-1}\tilde{\tau}_2} = \sigma. 
\]
This gives the proof of the surjectivity of $j_*^{-1}$ and hence 
of the theorem. 
\end{proof}

\section{Autoequivalence groups}

As in Section \ref{sec:4}, let $X_0$ be a smooth 
projective curve over a field $\mb k$ and 
let $\mca X$ be an 
infinitesimal deformation of $X_0$ over 
an Artinian local ring $(R, \mf m)$. 
The aim of this section is to show that 
the autoequivalence group $\Aut{\mb D^b(\mca X)}$ naturally 
acts on $\mb D^b(X_0)$.

\begin{lem}\label{lem:0211}
Let $\Phi \in \Aut {\mb D^b(\mca X)}$ and 
suppose that 
\[
\xymatrix{
	A \ar[r] & B \ar[r] & C \ar[r] & A[1]
}
\]
is a distinguished triangle in $\mb D^b(\mca X)$ with 
$A, B, C \in \mr{Coh}(X_0)$. 
If $\Phi(A)$ and $\Phi(C)$ lie in $\mr{Coh}(X_0)$, 
then so des $\Phi(B)$. 
\end{lem}

\begin{proof}
To complete the proof, it suffices to show the following: 
	\begin{enumerate}
	\item $\Phi(B) \in \mr{Coh}(\mca X)$, 
	\item the maximal ideal $\mf m$ annihilates $\Phi(B)$. 
\end{enumerate}
The first assertion follows immediately form the assumption on $A$ and $C$. 
Since $B\in \mr{Coh}(X_0)$,  
the ideal $\mf m$ annihilate $B$. 
Because $\Phi$ is $R$ linear, 
$\mf m$ also annihilates $\Phi(B)$. 
\end{proof}

\begin{lem}\label{lem:0304}
Assume that an equivalence $\Phi \in \Aut{\mb D^b(\mca X)}$ satisfies
\begin{equation*} \label{eq:0304}
\Phi (\mr{Coh}(X_0)) \subset 
\mr{Coh}(X_0)[k]
\end{equation*}
for some $k\in \bb Z$. 	
Then $\Phi$ preserves 
the subcategory $\Im j_*$ of $\mb D^b(\mca X)$. 
\end{lem}

\begin{proof}
By assumption, 
$\Phi$ sends every object in $\Im j_*$ to 
an object in $\Im j_*$. 
Thus it remains to show that 
$\Phi$ also preserves the morphisms in $\Im j_*$. 

Let $A$ and $B$ be indecomposable objects 
in $\Im j_*$. 
Then there exist sheaves $\mca F$ and $\mca G$ on $X_0$ 
and integers $m$ and $n$ such that 
\[
A= j_* \mca F[m] \text{ and } B=j_* \mca G[n].  
\]
Set $\Phi(A)=j_* \mca F'[m+k]$ and 
$\Phi(j_* B)= j_* \mca G'[n+k]$. 
Then we have an isomorphism
\begin{equation}
\label{eq:0310}	
\Hom_{\mca X}^p(A, B) \cong 
\Hom_{\mca X}^{p+n-m}(j_* \mca F, j_* \mca G) \cong 
\Hom_{\mca X}^{p+n-m}(j_* \mca F', j_* \mca G').  
\end{equation}

If a non-zero morphism 
$
e \in \Hom_{\mca X}^p(A, B) \cong 
\Hom_{\mca X}^{p+m-n}(j_* \mca F, j_* \mca G)
$
is in $\Im j_*$, 
the integer $p+n-m$ must be $0$ or $1$.

Suppose $p+m-n=0$. 
Since the right hand side 
in (\ref{eq:0310}) is isomorphic to 
$\Hom_{X_0}^0(\mca F', \mca G')$, 
the image $\Phi(e)$ is in $\Im j_*$. 
Suppose $p+m-n=1$. 
Lemma \ref{lem:0211} implies 
that $\Phi (e)$ is in $\Im j_*$. 
\end{proof}

\subsection{The case $g(X_0)=0$}

\begin{prop}\label{prop:linecase}
Let $\Phi \in \mb D^b(\bb P^1_R)$. 
Then $\Phi$ preserve the image 
of $j_* \colon \mb D^b(\bb P^1) \to \mb D^b(\bb P^1_R)$. 
\end{prop}

\begin{proof}
Let $\mb S$ be the 
class of objects in $\mb D^b (\bb P^1_R)$	satisfying 
the conditions in the assumption of Lemma \ref{lem:0213}. 
Any equivalence $\Phi \in \Aut{\mb D^b(\bb P^1_R)}$
preserves the class $\mb S$. 
By Lemma \ref{lem:260308}, 
we have 
\begin{align*}
\Hom^{1}_{\bb P^1_R}(j_* \mca O(n), j_* \mca O(n)) &\cong 
\mf m / \mf m^{2}, \text{ and } \\
\Hom^{1}_{\bb P^1_R}(j_* \mca O_x, j_* \mca O_x) & \cong 
\mb k  \+ \mf m / \mf m ^{2}. 
\end{align*}
Hence 
$\Phi$ never sent the skyscraper sheaf $j_* \mca O_x$ to 
the invertible sheaf $j_* \mca O(k)$ on $\bb P^1$, up to shifts. 

Thus 
we see the following: 
\begin{enumerate}
	\item For any point $x \in X_0$, 
there exist a closed point $y \in X_0$ and an integer $n$ 
such that $\Phi (j_* \mca O_x) = j_* \mca O_y[n]$. 
\item For any $k \in \bb Z$, 
there exist integers $\ell$ and $m$ such that 
$\Phi (j_* \mca O(k))  =  j_* \mca O(\ell)[m]$. 
\end{enumerate}
To complete the proof, 
it suffices to show that the integers $n$ and $m$ are 
independent of the choice of $x$ and $k$.

Fix $k$ as $0$. 
Now we have 
\[
\Hom_{\mca X} ^p
(j_* \mca O_{\bb P^1} , j_* \mca O_x) \cong 
\Hom _{\mca X}^p 
(\Phi (j_* \mca O_{\bb P^1}) , \Phi (j_* \mca O_x)) \cong  
\Hom_{\mca X}^p
(j_* \mca O(\ell), j_* \mca O_y[n-m]). 
\]
By 
Lemma \ref{lem:260308}, 
the right hand side is non-zero whenever 
$n-m+p \geq 0$. 
Again Lemma \ref{lem:260308} implies 
the left hand side is non-zero whenever $p \geq 0$. 
Thus we have $n-m=0$, 
and in particular, 
$n$ is constant for $x\in X_0$. 
Fixing a point $x$, the similar argument yields 
that $m$ is constant for $k$. 

As a result we obtain 
\begin{align*}
\Phi (j_* \mca O_x)= j_* \mca O_y[n], 
\text{ and } 
\Phi (j_* \mca O(k)) = j_* \mca O(\ell)[n]. 	
\end{align*}

Hence the equivalence $\Phi$ sends 
$\mr{Coh}(\bb P^1)$ to $\mr{Coh}(\bb P^1)[n]$  for some $n\in \bb Z$, 
and we obtain the desired assertion by 
Lemma \ref{lem:0304}. 
\end{proof}

\subsection{The case $g(X_0)\geq 2$}

\begin{lem}\label{lem:02031}
Let $E_0$ be a locally free sheaf on $X_0$. 
If $g(X_0) \geq 2$, then the following 
\begin{equation}\label{eq:0203}
\dim \Hom_{X_0} (E_0, E_0) < \dim \Hom_{X_0}^1(E_0, E_0)
\end{equation}
holds. 
In particular we have $\chi (E_0, E_0) <0$. 
\end{lem}

\begin{proof}
Let $K_{X_0}$ be the canoical divosr of $X_0$. 
By Serre duality, we have
\[
\Hom_{X_0}^1 (E_0, E_0) \cong 
\Hom_{X_0}(E_0, E_0(K_X))^{\vee}. 
\]
Moreover we have an inclusion 
\[
\Hom_{X_0}(E_0, E_0) \subset \Hom_{X_0}(E_0, E_0(K_X)). 
\]
Thus we obtain 
\begin{align*}
\chi(E_0, E_0) &= \dim \Hom(E_0, E_0) - \dim \Hom^1(E_0, E_0) \\
				&= \dim \Hom(E_0, E_0)- \dim \Hom (E_0, E_0(K_X))	 \leq 0. 
\end{align*}

Now we show that the equality above never hold. 
Suppose to the contrary, if the equality holds, 
we obtain the following short exact sequence: 
\[
\xymatrix{
0 \ar[r] & 
\Hom_{X_0}(E_0, E_0 (K_X)) \ar[r] & 
\Hom_{X_0}^1 (E_0, E_0) \ar[r] & 	
\Hom_{X_0}^1(E_0, E_0(K_X)|_{K_X}) \ar[r] & 0. 
}
\]
Since the left hand side is non-zero, we have 
\[
\dim \Hom^1_{X_0}(E_0, E_0) > \dim \Hom_{X_0}^1(E_0, E_0(K_X)). 
\]
Then Serre duality implies the inequality 
\[
\dim \Hom^0(E_0, E_0(K_X)) >\dim  \Hom ^0_{X_0}(E_0, E_0)
\]
which contradicts the assumption. 
Thus we have the inequality (\ref{eq:0203}). 
\end{proof}

\begin{rmk}\label{rmk:0204}
Under the same notation in Lemma \ref{lem:02031}, 
the following holds: 
	\[
\dim \Hom_{X_0}^1(E_0, E_0) >1	. 
	\]
\end{rmk}

\begin{lem}\label{lem:0212}
Let $\Phi \in \Aut{\mb D^b(\mca X)}$ and 
suppose $g(X_0) \geq 2$. 
\begin{enumerate}
	\item For any $x \in X_0$, there exist $y \in X_0$ and an integer $n$ such that 
	\[
	\Phi(j_* \mca O_x) = j_* \mca O_y[n]. 
	\]
	\item For any stable locally free sheaf $E$ on $X_0$, 
	there exist a stable locally free sheaf $F$ on $X_0$ and 
	an integer $n' $ such that 
	\[
	\Phi(j_* E) = j_* F [n'] . 
	\] 
	 \item The integers $n$ and $n'$ coincide and are independent of 
	 the choices of $x$ and $E$. 
\end{enumerate}
\end{lem}

\begin{proof}
	Since $g(X_0)\geq 2$, 
the objects $j_*\mca O_x$ and any stable 
locally free sheaf $E$ are stable 
for any $\sigma \in \Stab{\mb D^b(\mca X)}$. 
Hence the images of them via $\Phi$ is also stable 
with respect to any $\tau \in \Stab{\mb D^b(\mca X)}$.

Recall 
\[
\Stab{\mb D^b(\mca X)} \cong \Stab{\mb D^b(X_0)} \cong \widetilde{\mr{GL}}_2^+(\bb R). 
\]
Thus, for arbitrary $\tau \in \Stab{\mb D^b(\mca X)}$, 
a $\tau$-stable object must be 
a skyscraper sheaf $\mca O_y$, or 
a slope stable locally free sheaf $F$ on $X_0$. 
Hence $\Phi (j_* \mca O_x)$ must be 
$j_* \mca O_y$ or $j_* F$ up to shifts.

Comparing the dimensions of self-extensions, 
Lemma \ref{lem:260308}
\begin{align*}
\Hom^1_{\mca X}(j_* \mca O_x, j_* \mca O_x) &= \mb k \+ \mf m / \mf m^2	\text{ and } \\
\Hom^1_{\mca X}(j_* F, j_* F) &= \Hom^1_{X_0}(F, F) \+ \mf m/\mf m^2. 
\end{align*}
Since $\dim \Hom_{X_0}^1(F, F)>1$ by Remark \ref{rmk:0204}, 
these dimension do not coincide. 
Hence we obtain 
$\Phi(j_* \mca O_x) = j_* \mca O_y [n]$ and 
$\Phi (j_* E )= j_* F [n']$. 
Moreover a similar argument in the proof of 
Proposition \ref{prop:linecase} gives the proof 
of the assertion (3). 
\end{proof}

\begin{lem}\label{lem:0212-krso}
Let $\Phi \in \Aut{\mb D^b(\mca X)}$ and 
suppose $g(X_0) \geq 2$. 
\begin{enumerate}
	\item 
For any locally free sheaf 
$F$ on $X_0$, there exist a locally 
free sheaf $G$ on $X_0$ and an integer $n$ such that 
$\Phi(j_* F)=j_* G[n]$. 
Moreover the integer $n$ is independent of the choice 
of $F$.  
\item For any torsion sheaf $T_x$ supported 
on a point $x \in X_0$, 
there exist a point $y \in X_0$ and an integer $n$ such that 
$\Phi(j_* T_x) = j_* T_y [n]$. 
Moreover the integer $n$ is independent of the choice 
of $F$.  
\end{enumerate}
\end{lem}

\begin{proof}
Since the proof is similar, we only prove the assertion (1). 
We first prove the assertion 
when $F$ is slope semistable, that is, 
we show the following: 
\begin{description}
	\item[Claim]  For any slope semistable locally free sheaf 
$F$ on $X_0$, there exist a semistable  locally 
free sheaf $G$ on $X_0$ and an integer $n$ such that 
$\Phi(j_* F)=j_* G[n]$. 
\end{description}

Take a JH filtration of $F$: 
\[
0 = F_0 \subset F_1 \subset F_2 \subset \cdots \subset F_{m-1} \subset F_{m} =F . 
\]
Here the quotients $S_i:= F_i/F_{i-1}$ is 
slope stable locally free sheaf which has 
the same slope of $F$. 

By Lemma \ref{lem:0212},  each $\Phi(j_* S_i)$ 
is a slope stable locally free sheaf with the same slope, 
up to shifts. 
Without loss of generality, 
we may assume $\Phi(j_* S_i)$ is in $\mr{Coh}(X_0)$. 
Using Lemma \ref{lem:0211} inductively, we see 
that $\Phi(j_*  F)$ is a slope semistable locally 
free sheaf on $X_0$, 
and this gives the proof of the claim.

Let us discuss the general case. 
For any locally free sheaf on $X_0$, 
there exists the HN filtration of $F$. 
Then each quotient $F_i/F_{i-1}$ is slope semistable. 
By the claim above, 
$\Phi(j_* F_i/F_{i-1})$ is a slope semistable 
locally free sheaf on $X_0$. 
Hence $\Phi(j_*F)$ is a locally free sheaf on $X_0$ 
by Lemma \ref{lem:0211}
\end{proof}

\begin{prop}\label{prop:higher-genus}
Let $\Phi$ be in $\Aut{\mb D^b(\mca X)}$ and 
suppose 
$g(X_0) \geq 2$. 
Then $\Phi $ preserves the image $\Im (j_* \colon \mb D^b(X_0) \to \mb D^b(\mca X))$. 
\end{prop}

\begin{proof}
By Lemma \ref{lem:0212-krso}, 
we see that, up to shifts, 
any locally free sheaf and torsion sheaf 
on $X_0$ are sent to a locally free sheaf 
and to a torsion sheaf on $X_0$, respectively. 
Moreover, 
by the proof of Proposition \ref{prop:linecase}, 
the shifts is independent of the choice 
of the sheaf. 
Thus Lemma \ref{lem:0304} 
implies the desired assertion. 
\end{proof}

\subsection{The case $g(X_0)=1$}
From now on, we assume that $g(X_0)=1$.

\begin{lem}
Let $\mca F$ and $\mca G$ be a stable 
locally free sheaf on $X_0$. 
Then the following holds: 
\begin{enumerate}
	\item If $\mu(\mca F) > \mu (\mca G) $, then 
	$\Hom^1_{X_0}(\mca F, \mca G )\neq 0$ and 
	$\chi(\mca F, \mca G) <0$. 
	\item If $\mu(\mca F) < \mu (\mca G) $, then 
	$\Hom_{X_0}(\mca F, \mca G )\neq 0$ and 
	$\chi(\mca F, \mca G) >0$. 
	\item 
	If $\mca F \not\cong \mca G$ and 
	$\mu(\mca F) = \mu(\mca G)$, then $\Hom^1(F, G)=\Hom^1(G, F) =0$. 
\end{enumerate}
\end{lem}

\begin{proof}
We prove the assertion (1). 
Since they are stable, we have $\Hom(\mca F, \mca G)=0$. 
If $\Hom^1(\mca G, \mca F)=0$, the Serre duality yields 
$\Hom(\mca F, \mca G)=0$ which gives a contradiction. 
Hence we have $\Hom^1(\mca F, \mca G)=0$ and $\chi(\mca F, \mca G)<0$. 
The proof of the assertion (2) is similar. 

For the assertion (3), 
Riemann-Roch theorem yields $\chi(\mca F, \mca G)=0$
by the assumption $\mu(\mca F)=\mu(\mca G)$. 
Since $\mca F$ and $\mca G$ are not isomorphic with the same slope, 
we have $\Hom(\mca F, \mca G)=0$. 
\end{proof}

\begin{lem}\label{lem:260310}
Let $\mca F$ be a locally free sheaf on $X_0$. 
If $\mca F$ is indecomposable then $\mca F$ is semistable and 
the stable factor of the JH filtration is unique. 
\end{lem}

\begin{proof}
The semistability follows from Serre duality. 
Suppose that $\mca F$ has at least two stable factor, 
say $A$ and $B$.

Choose a semistable subsheaf $\mca G\subset \mca F$ 
whose Jrdaon-Holder factors are all isomorphic to $A$ and 
such that 
$\Hom(A, \mca F/ \mca G)=0$. 
Then, 
we have $\chi(A, \mca F/\mca G)=0$, 
since the slopes of $A$ and of $\mca F/\mca  G$. 
Thus we obtain $\Hom(\mca F/\mca G, A)=0$ by Serre duality. 
By the assumption for $\mca G$, 
we have $\Hom^1(\mca F/\mca G , \mca G)=0$, 
and this implies that $\mca G$ is a direct summand of $\mca F$. 
Hence $\mca G$ must bet zero, and gives the proof. 
\end{proof}

\begin{lem}\label{lem:0309}
	Let $\mca F$ and $\mca G$ be 
	indecomposable sheaves on $X_0$. 
	If $\mu(\mca F) >  \mu(\mca G)$, 
	then the following holds:
\begin{align}
	\Hom_{X_0}^1(\mca F, \mca G) & \cong 
			\Hom_{\mca X}^1(j_* \mca F, j_* \mca G). \label{eq:03092}
\end{align}
\end{lem}

\begin{proof}
By the assumption we have 
$
\Hom_{X_0}(\mca F, \mca G)=0 $ 
 and 
$\Hom_{X_0}^1(\mca F, \mca G) \neq 0$. 
Then Lemma \ref{lem:260308} yields 
the isomorphism (\ref{eq:03092}). 
\end{proof}

\begin{lem}\label{lem:genus1key}
Let $\Phi \in \Aut{\mb D^b(\mca X)}$. 
Assume that 
	 there exists a point $x_0 \in X_0$ such that 
	$\Phi(j_* \mca O_{x_0})= j_* \mca E_0[1]$, where 
	$\mca E_0$ is a locally free sheaf on $X_0$.  
	Then the following holds. 
\begin{enumerate}
	\item 
	For any point $y \in X_0$, 
	there exists a stable locally free sheaf 
	$\mca E_y$ on $X_0$ such that 
	$\Phi(j_* \mca O_y)= j_* \mca E_y[1]$.  
	\item For any point $y \in X_0$, 
	there exists a stable locally free 
	sheaf $\mca G_y$ on $X_0$ such that 
	$\Phi(j_* \mca G_y)=j_* \mca O_y$. 
	\item Let $\mca F$ be an indecomposable locally free sheaf on $X_0$. 
	\begin{itemize}
		\item 
	If $\mu(\mca F) > \mu(\mca G_{x_0})$, then we have 
	$\Phi(j_* \mca F)=j_* \mca F'[1]$, where 
	$\mca F'$ is a locally free sheaf on $X_0$. 
	\item 
	If $\mu(\mca F) < \mu(\mca G_{x_0})$, then we have 
	$\Phi(j_* \mca F)=j_* \mca F'$, where 
	$\mca F'$ is a locally free sheaf on $X_0$. 
	\end{itemize}
\end{enumerate}
\end{lem}

\begin{proof}
We show the assertions (1) and (2) simultaneously. 
	Since the skyscraper sheaf 
$j_* \mca O_y$ is stable for any $\sigma \in \Stab{\mb D^b(\mca X)}$, 
$\Phi(j_* \mca O_y)$ is also stable. 
By the classification of stable objects, 
$\Phi(j_* \mca O_y)$ must be 
a skyscraper sheaf $j_*\mca O_z$ or 
a stable locally free sheaf $j_* \mca E_y$, up to shifts. 
Suppose that $\Phi(j_* \mca O_y)=j_* \mca O_z[k]$ 
for some $k \in \bb Z$. 

Then we have 
\begin{align*}
	 \Hom_{\mca X}^p(j_* \mca O_x, j_* \mca O_y) 
	 \cong \Hom_{\mca X}^p(j_* \mca E_0[1], j_* \mca O_z[k]). 
\end{align*}
If $x \neq y$, the left hand side is zero for any $p\in \bb Z$. 
However the right hand side is non-zero for some $p \in \bb Z$. 
Hence 
for any $y\in X_0$, there exists an integer $n_y$ such that 
$\Phi(j_* \mca O_y )= j_* \mca E_y[n_y]$. 

Then for any $y \in X_0$, 
there exist an stable locally free sheaf $\mca E_y$ on $X_0$ 
and an integer $m_y$ 
such that $\Phi(j_* \mca E_y) = j_* \mca O_y[m_y]$. 
By Lemma \ref{lem:260308}, 
$\Hom_{\mca X}^p(j_* \mca G_y, j_* \mca O_{x_0})$ is non-zero if 
and only if $p \geq 0$. 
The equivalence $\Phi$ yields 
\[
\Hom_{\mca X}^p(j_* \mca G_y, j_* \mca O_{x_0})
\cong 
\Hom_{\mca X}^{p+1-m_y}
(j_*\mca O_y, j_* \mca E_0). 
\]
By Lemma \ref{lem:260308}, 
the right hand side is non-zero if and only if 
$p+1-m_y \geq 1$. 
Hence we have $m_y=0$ and this implies 
the assertion (2). 

Applying the assertion (2) as $y=x_0$, 
we have 
$\Hom_{\mca X}^p (j_*\mca G_{x_0}, j_* \mca O_y)$ 
is non-zero if and only if $p\geq 0$. 
Since $\Phi$ is an equivalence we have 
\[
\Hom_{\mca X}^p(j_* \mca G_{x_0}, j_* \mca O_{y})
\cong 
\Hom_{\mca X}^{p+n_y}
(j_*\mca O_{x_0}, j_* \mca G_y). 
\]
Since the right hand side is non-zero if and only 
if $p+n_y \geq 1$, we see $n_y=1$ for all $y$ and 
this complete the proof of the assertion (1).

To compete the proof, 
let $\mca F$ be an indecomposable 
locally free sheaf on $X_0$. 
By Lemma \ref{lem:260310}, 
$\mca F$ is given by a finite 
successive extension of a single stable locally free 
sheaf $\mca F_0$. 
By Lemma \ref{lem:0211}, 
it suffices to show the assertion for the case 
$\mca F=\mca F_0$. 
Since $\mca F_0$ is stable, 
there exist a stable locally free sheaf $\mca F'$ on $X_0$ 
and an integer $n$ such that 
$\Phi(j_*\mca F_0)= \mca F'[n]$.

Suppose $\mu(\mca F_0) > \mu(\mca G_{x_0})$. 
Then we have 
$\Hom_{X_0}^0(\mca F_0, \mca G_{x_0})=0$ and 
$\Hom_{X_0}^1(\mca F_0, \mca G_{x_0})\neq 0$.     
Hence 
$\Hom_{\mca X}^p(j_* \mca F_0, \mca G_{x_0})$ is non-zero 
if and only if $p \geq 1$ by Lemma \ref{lem:260308}
The equivalence $\Phi $ yields the isomorphism: 
\[
\Hom_{\mca X}^p(j_* \mca F_0, j_* \mca G_{x_0}) 
\cong 
\Hom_{\mca X}^{p-n} (j_* \mca F', j_* \mca O_x)
\]
Similarly to the argument above, we see $n=0$. 
The proof of remained case is similar. 
\end{proof}

\begin{prop}\label{prop:elliptic}
Let $X_0$ be a smooth projective curve over a 
field $\mb k$ with genus 1. 
Then any equivalence $\Phi \in \Aut{\mb D^b(\mca X)}$ 
preserves the image of 
$j_* \colon \mb D^b(X_0) \to \mb D^b(\mca X)$. 
\end{prop}

\begin{proof}
If an object in $\Im j_*$ is indecomposable, 
then it must be semistable sheaf on $X_0$ 
with a unique stable factor by Lemma \ref{lem:260310}. 

Suppose that 
for any skyscraper sheaf $j_* \mca O_x$, 
$\Phi(j_* \mca O_x)$ is a skyscraper sheaf $j_* \mca O_y$ up to shifts. 
are sent a skyscraper sheaf $j_* $
Then, for any indecomposable locally free sheaf 
$\mca E$, 
$\Phi(j_* \mca E)$ is a locally free sheaf on $X_0$ up to shifts. 
Similarly to the proof of Proposition \ref{prop:higher-genus}, 
we see that 
$\Phi$ preserves the image $\Im j_*$.

Suppose that 
there exists a point $x\in X_0$ such that 
$\Phi (j_* \mca O_x)$ is not a shift of 
a skyscraper sheaf. 
Without loss of generality, we may assume $\Phi(j_*\mca O_x)[-1]$ is 
a locally free sheaf on $X_0$. 
Then by Lemma \ref{lem:genus1key}, we see that 
the image of any indecomposable sheaf on $X_0$ via $\Phi$ 
is a sheaf on $X_0$. 
Hence $\Phi$ preserves the objects in the subcategory $\Im j_*$.

To complete the proof, it suffices to show that 
$\Phi$ preserves morphisms in $\Im j_*$. 
To show this, 
let $(A, B)$ be a pair 
of indecomposable (not necessary locally free) 
coherent sheaves on $X_0$.

Suppose that $\mu(A) >  \mu(B)$. 
Then, by Lemma \ref{lem:0309}, we have 
\begin{align}
\label{eq:031012}
	\Hom_{\mca X}^0(j_* \mca A, j_* \mca B)&=0 , \text{ and } \\
\label{eq:031011}
	\Hom_{\mca X}^1(j_* \mca A, j_* \mca B)&\cong \Hom_{X_0}^1(A, B) \neq 0. 
\end{align}
For the image of objects, Lemma \ref{lem:genus1key} yields 
the following three cases. 
\[
(\Phi(j_* A), \Phi(j_*B)) = 
\begin{cases}
	(j_* E[1], j_* F[1]) \\
	(j_* E[1], j_*F) \\
	(j_* E,j_* F). 
\end{cases}
\]
Here $E$ and $F$ are indecomposable 
sheaves on $X_0$. 
In the first and third cases, 
the vanishing (\ref{eq:031012}) implies 
$\Hom_{\mca X}(j_* E, j_*F)=0$  and hence we have 
$\mu(E) > \mu(F)$. 
Then Lemma \ref{lem:0309} and the equivalence $\Phi$ imply
\[
\Phi \colon 
\Hom_{\mca X}^1 (j_*  A, j_* B) \to
\Hom_{\mca X}^1(j_* E, j_* F)  
\cong 
\Hom_{X_0}^1(E, F) .  
\]
Hence the image of morphism $f \colon j_* A \to j_*B[1]$ by $\Phi$ 
belongs to $\Im j_*$. 
In the second case, 
the equivalence $\Phi$ yields the map 
\[
\Phi \colon 
\Hom_{\mca X}^1 (j_*  A, j_* B) \to
\Hom_{\mca X}^0(j_* E, j_* F)  
\cong 
\Hom_{X_0}^0(E, F) .  
\]
Again, the image of any morphism $f$ 
is in the image $\Im j_*$.

Suppose that $\mu(A)< \mu(B)$.  
Since we have $\Hom_{X_0}^1(A,B)=0$, 
a morphism $e \colon j_*A \to j_*B [1]$ is in $\Im j_*$ if 
and only if $e$ is zero. 
Hence the image $\Phi(e)$ is in $\Im j_*$. 

Now Lemma \ref{lem:genus1key} again yields the following 
three cases. 
\[
(\Phi(j_* A), \Phi(j_*B)) = 
\begin{cases}
	(j_* E[1], j_* F[1]) \\
	(j_* E, j_*F[1]) \\
	(j_* E,j_* F). 
\end{cases}
\]

In the first and third cases, 
 the image of any morphism $j_* A \to j_* B$ by $\Phi$ 
is clearly $\Im j_*$. 
In the second case, 
the vanishing 
$\Hom^{-1}_{\mca X}(j_*A, j_* B)=0$ and 
the equivalence $\Phi$ yields the vanishing 
\[
\Hom_{\mca X}^0(j_* E, j_* F)=0. 
\]
Hence $\mu(E)>\mu(F)$ must hold. 
Then Lemma \ref{lem:0309} and $\Phi$ induce 
\[
\Phi \colon \Hom_{\mca X}^0(j_*A, j_* B ) \to 
\Hom _{\mca X}^1(j_* E, j_*F) \cong 
\Hom_{X_0}^1(E, F). 
\]
Thus the image of any morphism $f \colon j_*A \to j _* B$ 
by $\Phi$ belongs to $\Im j_*$. 

Finally suppose that $\mu(A)=\mu(B)$. 
By Riemann-Roch theorem, we see 
\[
\Hom_{\mca X}^0(j_* A, j_* B) =0 \iff 
\Hom_{\mca X}^1(j_* A, j_* B) =0. 
\]
Moreover, if $\Hom_{\mca X}(j_*A, j_*B)=0$, then 
nothing to prove. 
If $\Hom^0_{\mca X}(j_* A, j_*B) \neq 0$, then 
the stable factors of $A$ and $B$ must be the same. 
Hence, by Lemma \ref{lem:0211}, 
we see that 
the image of any morphism $j_* A \colon j_*B [1]$ 
is in $\Im j_*$. 
Thus we complete the proof. 
\end{proof}

Combining the above Propositions \ref{prop:linecase}, 
\ref{prop:higher-genus}, and \ref{prop:elliptic}, 
we obtain the following: 
\begin{prop}
Let 
 $\mca X \to \Spec R$ be an infinitesimal deformation 
of a smooth projective curve $X_0$. 
There exists a natural homomorphism 
of groups:
\begin{equation*}
\label{eq:0210}	
j_*^{-1} \colon 
\Aut {\mb D^b(\mca X)} \to \Aut{\mb D^b(X_0)}
\end{equation*}
\end{prop}

%
%

\end{document}